\documentclass[12pt,a4paper,withmarginpar]{article}

\usepackage[utf8x]{inputenc}
\usepackage[T1]{fontenc}
\usepackage[a4paper,top=3cm,bottom=3cm,left=3cm,right=3cm,marginparwidth=1.75cm]{geometry}

\usepackage{verbatim}
\usepackage{enumitem}
\input{xy}
\xyoption{all}
\usepackage[all]{xy} %%rysowanie diagramow
\usepackage{amsmath}
\usepackage{graphicx}
\usepackage{mathtools}
\usepackage[colorinlistoftodos]{todonotes}
\usepackage[colorlinks=true, allcolors=blue]{hyperref}
\usepackage{amssymb}
\usepackage{amsthm}
\theoremstyle{plain}
\newtheorem{thm}{\textsf{\textbf{Theorem}}}[section]
\newtheorem{lem}[thm]{\textsf{\textbf{Lemma}}}

\newtheorem*{thm*}{\textsf{\textbf{Theorem}}}

\theoremstyle{definition}
\newtheorem{dfn}[thm]{\textbf{\textsf{Definition}}}

\newtheorem{convention}[thm]{\textbf{\textsf{Convention}}}
\newtheorem{rem}[thm]{{\textsf{Remark}}}
\newtheorem{ex}[thm]{{\textsf{Example}}}

\newcommand{\proj}{proj}

\newcommand{\Wee }{\mathcal W}

\newcommand{\nat }{\mathbb N}
\newcommand{\Uee }{\mathcal U}

\DeclareMathOperator{\Ker}{Ker}

\DeclareMathOperator{\im}{Im}

\DeclareMathOperator{\per}{Per}
\DeclareMathOperator{\map}{Map}
\DeclareMathOperator{\iso}{Iso}
\DeclareMathOperator{\id}{id}

\DeclareMathOperator{\Sym}{Sym}
\DeclareMathOperator{\N}{N}

\DeclareMathOperator{\stab}{Stab}
\title{Entropy of Toeplitz systems over residually finite groups}
\author{Przemysław Kucharski\footnote{AGH University of Science and Technology, Faculty of Applied Mathematics, al. Mickiewicza 30, Krak\'{o}w, Poland, e-mail: pkuchars@agh.edu.pl, ORCID iD: \url{https://orcid.org/0000-0002-3826-5827} 
}}
\begin{document}
		\maketitle
\begin{abstract}
	The purpose of this work is to bound sofic topological entropy of Toeplitz systems over residually finite groups and to prove the Krieger Theorem about attaining arbitrary entropy by the Toeplitz systems. To achieve these results, we discuss certain properties of the sofic topological entropy in the context of finitely indexed normal subgroups of the group. It will help us to formulate results almost independently of the natural sofic approximation sequences of residually finite groups.
\end{abstract}

The undeniable fact is that computing sofic entropy appears to be much harder than amenable entropy. It is visible in theorems tackling general properties, as they are exponentially more technical than their classical counterparts, for example compare proofs of the variational principle for the sofic case \cite[Chapter 10.10]{KerrLi} and for the amenable case \cite[Chapter 9.10]{KerrLi}. We can see this phenomenon also during concrete computation, for instance \cite[Prop. 10.28, Prop. 10.29]{KerrLi}. It is probably one of the reasons why there are so few examples of direct computations of sofic topological entropy and most of them are trivial or require additional structure on the space. The purpose of our article is to fill in this gap in the most efficient way, that is by computing sofic topological entropy for large class of dynamical systems - Toeplitz systems defined over residually finite groups, which, as we know, includes non amenable free groups. We will prove in Theorem \ref{krieger-rf}, that for some specific sofic approximation sequences, intimately connected with the group structure, for any given non negative real number $\alpha$ we can find a Toeplitz system with entropy that equals to $\alpha$. In other words, we will prove that the Krieger's Theorem holds for any residually finite group. To the best of our knowledge it will also be the first example of a family of dynamical systems over residually finite group, for which any sofic topological entropy can be attained. The Krieger Theorem for amenable residually finite groups was first proved in 2007 by Fabrice Krieger \cite{KriegerTheorem} by direct construction and then using indirect methods by Martha \L\k{a}cka and Marta Straszak \cite{MM2018}, but it was unknown till now, if it holds for arbitrary residually finite group. We explicitly construct desired Toeplitz system, hence our approach is closer to that of Krieger, although in the case of residually finite groups that are not amenable, one needs to better control each step of the induction. We also establish bounds for the sofic topological entropy in Theorem \ref{main}, which are the best possible in the light of Theorem \ref{toeplitz-example}.

All the main theorems are contained in Section \ref{Entropy of Toeplitz systems} - "Entropy of Toeplitz systems". Sections \ref{Amenable groups and sofic groups} and \ref{sofic topological entropy} are dedicated to definitions of sofic groups and sofic entropy following \cite[Chapter 10]{KerrLi}. On the other hand, Sections \ref{Dependency of entropy on sofic approximation} and \ref{Entropy of subshifts over residually finite group} contain necessary discussion of dependency of entropy on sofic approximation and introduction of key tools in computing entropy of symbolic systems. Sections \ref{Amenable groups and sofic groups} and \ref{sofic topological entropy} are entirely based on \cite[Chapter 9, 10]{KerrLi}. 

%%Author would like to thank Dominik Kwietniak and Martha \L\k{a}cka for many fruitful discussions and constant support.
\section{Preliminaries}\label{preliminaries}
By an action of the group $G$ on a set $X$ we mean a map $\beta\colon G\times X\to X$ that satisfies the usual conditions
\begin{enumerate}[label=(D\arabic*)]
	\item for every $x\in X$ we have $\beta(e,x)=\id$, where $e$ is the neutral element of $G$
	\item for every $s,t\in G$ we have $\beta(t,\beta(s,x))=\beta(ts,x)$.
\end{enumerate} A set $X$ with an action of $G$ is said to be a $G$-set. For brevity, we will usually write $sx$ instead of $\beta(s,x)$ and the function $\beta$ will not be named. If additionally $X$ is a topological space and $G$ is a topological group and the action $\beta$ is continuous, we will call the pair $(X,G)$ a dynamical system. Note that if $\beta$ is continuous and $X$ is compact, then $\beta(s,\cdot)\colon X\to X$ is automatically a homeomorphism for any $s\in G$. A pair $(Y,G)$ is a subsystem of $(X,G)$ if $Y$ is a closed invariant subset of $X$, that is $gY\subseteq Y$ for every $g\in G$. Below we present one of the most classical examples of a dynamical system. For simplicity of notation we denote by $k$ the set $\{1,\ldots,k\}$ for any $k\in\nat=\{1,2,3,...\}$.

Clearly, if $G$ is a discrete topological group, then $\beta$ is continuous iff $\beta(s,\cdot)\colon X\to X$ is continuous for any $s\in G$.

\begin{dfn}\label{factor-def}
	A homomorphism of $G$-sets $U$ and $V$ is a mapping $\chi\colon U\to V$ that commutes with the action of $G$, that is $g\chi(u)=\chi(gu)$ for every $g\in G$ and $u\in U$. A surjective homomorphism will be called a factor map. If the sets $U$ and $V$ are topological spaces, then $\chi$ is additionally required to be continuous. If $\chi$ is a bijection, then it will called an isomorphism of $G$-sets.
\end{dfn}
\begin{convention}
	To avoid any misunderstandings, we adopt the following convention for suprema and infima of subsets of $\mathbb{R}$
	$$
	\inf \emptyset =\infty~~\text{ and }~~\sup\emptyset=-\infty.
	$$
\end{convention}
On the product of topological spaces we will always consider a product topology.
\begin{lem}\label{k^G-dyn-sys}
	A pair $(k^{G},G)$ together with the action defined by $(gx)_{h}=x_{g^{-1}h}$ is a dynamical system.
\end{lem}
\begin{dfn}
	We will call $(k^{G},G)$ a full $G$-shift with base $k$. Any subsystem of $(k^{G},G)$ will be called a $G$-shift with base $k$ or just a $G$-shift, when no confusion can arise. 
\end{dfn}
\section{Amenable groups and sofic groups}\label{Amenable groups and sofic groups}
	As we have already noted, there are significant differences between the amenable entropy theory and the sofic entropy. The primary reason behind this phenomenon is existence of F\o{}lner sequences in every amenable group, which provides methods of generalizing classical averaging arguments from the theory of $\mathbb{Z}$-action to the setting of amenable groups. Therefore, amenable groups are those for which some kind of a mean can be defined. As our main purpose is to develop tools in the setting of sofic groups, we present definitions and examples for both classes of groups for the sake of comparison.
	\begin{dfn}\label{amenabledfn}
		Let $G$ be a countable group. A sequence $\{F_{n}\}_{n\in\nat}$ of finite nonempty subsets of $G$ is \emph{a F\o{}lner sequence} if for every element $s\in G$ we have 
		\begin{equation}\label{amenable}
		\lim_{n\to\infty}\frac{|sF_{n}\Delta F_{n}|}{|F_{n}|}=0.
		\end{equation}
		Group $G$ is \emph{amenable} if it admits a F\o{}lner sequence. 
	\end{dfn}
	\begin{ex}
		For any $k\in\nat$, the group $\mathbb{Z}^{k}$ is amenable. We define its F\o{}lner sequence to be the one consisting of cubes increasing in diameter, that is $F_{n}=[-n,n]^{k}$, for $ n=1,2,... $. It is easy to check that $\{F_{n}\}_{n\in\nat}$ indeed satisfies condition \eqref{amenable}.
	\end{ex}
	\begin{dfn}
		Let $G$ be a countable group and $F$ be its finite subset. Fix $k\in\nat$. For a $G$-shift $X\subset k^{G}$, we define a set of all words based on $F$ by the formula
		\[ 
		\mathcal{B}_{F}(X):=\{x\in k^{F}~|~x=y|_{F}\text{ for some }y\in X\}.
		 \]
	\end{dfn}
	\begin{dfn}
		For an amenable group $G$, positive natural number $k$ and a $G$-shift $X\subset k^{G}$ we define the entropy of $(X,G)$ to be the limit
		\begin{equation}\label{amenable-entropy}
		h(X,G)=\lim_{n\to\infty}\dfrac{1}{|F_{n}|}\log |\mathcal{B}_{F_{n}}(X)|.
		\end{equation}
		It can be proved that the above limit always exists and the entropy defined in this way coincides with the classical topological entropy for $\mathbb{Z}$-shifts \cite[Chapter 9]{KerrLi}.
	\end{dfn}
	\begin{convention}
		Let $V$ be an arbitrary set. We denote by $\Sym(V)$ the group of all bijections $V\to V$ with multiplication given by composition.
	\end{convention}
	\begin{dfn}\label{sofic-def}
		We call a countable group $G$ \emph{sofic} if there exist sequences $\{V_n\}_{n=1}^{\infty}$ of finite sets and $\Sigma=\{\sigma_{n}\colon G\to \Sym(V_n)\}_{n=1}^{\infty}$ of mappings, such that $\Sigma$ is asymptotically multiplicative and asymptotically free, meaning that
	\begin{enumerate}[label=(S\arabic*)]
		\item\label{S1}$\lim_{n\to\infty} |\{v \in V_{n} : \sigma _{n,st} (v) = \sigma _{n,s} \sigma _{n,t} (v)\}| /|V_{n}| = 1$ for all $s, t \in G$, and
		\item\label{S2} $\lim_{n\to\infty} |\{v \in V_{n} : \sigma _{n,t} (v) \neq \sigma _{n,s}(v) \}| /|V_{n}| = 1$ for all distinct $s, t \in G$,
	\end{enumerate} 
	where $\sigma_{n,s}$ denotes the image of a group element $s$ under $\sigma_{n}$. Such a sequence $\Sigma $ is called \emph{a sofic sequence} or a \emph{a sofic approximation}. If mappings $\sigma_{n}$ are group homomorphisms we call $\Sigma$ \emph{a sofic approximation by homomorphisms}.
	\end{dfn}
	\begin{rem}
		Every amenable group $G$ is also a sofic group. If $\{F_{n}\}_{n\in\nat}$ is a F\o{}lner sequence for $G$, define $\sigma_{n}\colon G\to\Sym(F_{n})$ by 
		$$
		\sigma_{n,s}(f)=
		\left\{ \begin{array}{ll}
		sf & \textrm{ if }f\in F_{n}\cap s^{-1}F_{n},\\
		\alpha_{n}(f) & \textrm{ otherwise, }
		\end{array} \right.
		$$ where $\alpha_{n}\colon F_{n}\setminus s^{-1}F_{n}\to F_{n}\setminus sF_{n}$ is any bijection. Since $|F_{n}\setminus sF_{n}|\to 0$, the mappings $\sigma_{n}$ will satisfy conditions  \ref{S1} and \ref{S2}.
	\end{rem}
\begin{dfn}[see \cite{geom-group-theory}]	
	Let $S$ be a set. The group $(\langle S\rangle,\cdot)$, often abbreviated to $\langle S\rangle$, will be called \emph{a free group generated by} $S$. 
\end{dfn}
\begin{dfn}
	Group $G$ will be called \emph{a residually finite group} if there exists a decreasing sequence $H_{1}\supset H_{2}\supset\ldots$ of subgroups of finite index with trivial intersection $\bigcap_{n\in\mathbb{N}}H_{n}=\{e\}$. A decreasing sequence of subgroups with trivial intersection will be shortly denoted by $H_{n}\searrow \{e\}$.
	Additionally, a sequence of fundamental domains $\{F_{n}\}_{n\in\nat}$ corresponding to  $\{G/H_{n}\}_{n\in\nat}$ will be called a telescoping sequence of fundamental domains if $
	F_{n+1}=(F_{n+1}\cap H_{n})F_{n}$ and $e\in F_{n}$ for every $n\in\nat$.
\end{dfn}
\begin{rem}
	Note that if $H_{n}\searrow \{e\}$ for a group $G$, then we can also find a sequence $\{K_{n}\}_{n\in\nat}$ of normal subgroups of $G$ of finite index such that $K_{n}\searrow \{e\}$. Indeed, define $K_{n}=\bigcap_{g\in G}g^{-1}H_{n}g$. Clearly, every $K_{n}$ is a subgroup of $G$. The number of conjugacy classes of any finite index subgroup, say $K\subseteq G$, is less than its index and in fact is equal to the index of normalizer, which contains $K$ as a subgroup. Consequently, it follows from the inequality $[G:H\cap K]\leq [G:K][G:H]$ holding for any subgroups $H,K\subseteq G $, that $K_{n}$ is a finitely indexed normal subgroup with $K_{n}\searrow \{e\}$.
\end{rem}
\begin{ex}\label{example-res-fin}
 	Let $G$ be a countable residually finite group. Therefore there exists a decreasing sequence of subgroups of finite index $\{H_n\}_{n=1}^{\infty}$ intersecting on the neutral element. We define the sequence $\Sigma=\{\sigma_{n}\colon G\to \Sym(G/H_n)\}_{n\in\nat}$, where $G/H_n=\{gH_{n}\}_{g\in G}$, by $\sigma_{n,g}(cH_n):=cg^{-1}H_n$, for $c,g\in G$. It is easy to check that $\Sigma$ is indeed a sofic sequence for a group $G$. We will usually call it the sofic sequence with natural action on cosets of $H_{n}$. Note that we have not assumed $\{H_{n}\}$ is a sequence of normal subgroups. 
 \end{ex}
 \section{Sofic topological entropy}\label{sofic topological entropy}
	We now proceed with the definition of sofic topological entropy or briefly sofic entropy. Let $(X,G)$ be a compact metrizable dynamical system and $d$ be a continuous pseudometric on $X$. By compactness of $X$ we can assume that $d(x,y)\leq 1$ for any $x,y\in X$. From now on, if not stated otherwise, $\Sigma=\{\sigma_{n}\colon G\to \Sym(V_{n})\}_{n\in\nat}$ will be a sofic approximation of a group $G$. All the definitions presented in this section can be found in \cite[Chapter 10]{KerrLi}.
	\begin{dfn}
	For a finite set $V$ we define pseudometrics $d_{2} $ and $d_{\infty}$ on the set of all maps $V \to X$ by
	$$d_{2}(\varphi,\psi)=\Big( \frac{1}{|V|}\sum_{v\in V}(d(\varphi(v),\psi(v))^{2}\Big)^{1/2},$$
	$$d_{\infty}(\varphi, \psi)=\max_{v\in V}d(\varphi(v),\psi(v)).$$	
	\end{dfn}
	\begin{dfn}
	Let $F$ be a finite subset of $G$, $\delta>0$, and let $\sigma\colon G\to\Sym(V)$ for some finite set $V$. We define $\map(d, F, \delta,\sigma)$ to be the set of all maps $\varphi \colon V\to X$ such that $d_{2}(\varphi\sigma_{s},\alpha_{s}\varphi)\leq \delta$ for all $s\in F$, where $\alpha_{s}$ denotes the transformation $x\mapsto sx$ of $X$.
	\end{dfn}
	\begin{dfn}
	Given a pseudometric $d$ on a set $Y$ we write $\N_{\epsilon}(Y, d)$ for the maximum cardinality of a subset $E$ of $Y$ which is $(d, \epsilon)$-seperated in the sense that $d(y,z) >\epsilon$ for all distinct $y,z\in E$.
	\end{dfn}
	\begin{dfn}\label{entropy-def1}
	For a continuous pseudometric $d$ on $X$ we set
	$$
	h_{\Sigma}(d)=\sup_{\epsilon>0}\inf_{F\subset G}\inf_{\delta>0}\limsup_{n\to\infty} \frac{1}{|V_{n}|} \log \N_{\epsilon}(\map(d, F,\delta,\sigma_{n}),d_{\infty}),
	$$
	where the first infimum is over all finite sets $F\subset G$. We set $h_{\Sigma}(d)=-\infty$ if $\map(d, F,\delta,\sigma_{n})=\emptyset$ for all $n\in\nat$ big enough.
	\end{dfn}
		It turns out that instead of measuring pseudoorbits separation in terms of $d_{\infty}$ pseudometric  we can use $d_{2}$ pseudoometric.
	\begin{thm}{\cite[Prop. 10.23]{KerrLi}}\label{KerrLi}\label{d2-dinfty}
		If $d$ is a continuous pseudometric on $X$, then
		\begin{equation}\label{eq1}
		h_{\Sigma}(d)=\sup_{\epsilon>0}\inf_{F}\inf_{\delta}\limsup_{n\to\infty} \frac{1}{|V_{n}|} \log \N_{\epsilon}(\map(d, F,\delta,\sigma_{n}),d_{2}),
		\end{equation}
		where limits are taken as in Definition \ref{entropy-def1}.
	\end{thm}
	In general, the value $h_{\Sigma}(d)$ depends on the pseudometric $d$, but in the case when $d$ is \emph{dynamically generating}, that is for every distinct $x,y \in X$ there exists $s\in G$ with $d(sx,sy)>0$, then $h_{\Sigma}(d)$ coincides with $h_{\Sigma}(d')$ for any other continuous dynamically generating pseudometric $d'$.
	\begin{dfn}
		We define the \emph{sofic topological entropy} of a dynamical system $(X,G)$ with respect to $\Sigma$, often simplified to \emph{sofic entropy} of a dynamical system $(X,G)$ with respect to $\Sigma$, as the common value of $h_{\Sigma}(d)$ over all dynamically generating continuous pseudometrics $d$ on $X$.
	\end{dfn}
\section{Dependency of entropy on sofic approximation}\label{Dependency of entropy on sofic approximation}
In case of non amenable groups, it is very hard to compute sofic entropy with respect to arbitrary sofic approximation. There are also very few examples of computed entropy of non amenable group actions, where the sofic approximation is not explicitly known. This is because it is hard to determine how in those cases sofic sequence approximates group action. Despite all of that there are some results concerning dependency of the sofic entropy on sofic approximation. In a case, where two sofic approximations give the same entropy, we will call them equivalent. It seems that the most common, or maybe even the only, approach to the problem of distinguishing equivalent sofic approximations is that funded on edit distance \cite[Paragraph 2.2.4.]{Bowen2019}. But we will not follow it, since we believe that the one presented below is more suitable for formal proofs, that do not appeal to the expert's intuition, and above all our approach seems to have potential to be a funding ground for better tools than edit distance.
For any set $V$, we denote by $\amalg_{k}V$ the disjoint sum of $k$ copies of $V$, that is $\amalg_{k}V=(\{1\}\times V)\cup...\cup(\{k\}\times V)$. In this notation, let $\iota\colon\amalg_{k}V\to k $ be the projection on the first variable and $\kappa\colon\amalg_{k}V\to V $ be the projection on the second variable. In the same way we can define the disjoint sum of sets $V_{1},...,V_{l}$, for some $l\in\nat$, that is $\amalg_{i=1}^{l}V_{i}=(\{1\}\times V_{i})\cup...\cup(\{k\}\times V_{i})$.
\begin{lem}\label{lem-amalg}
	If $\Sigma=\{\sigma_{n}\colon G\to\Sym(V_{n})\}$ is a sofic approximation and $(a_{n})_{n\in\mathbb{N}}$ is a sequence in $\nat$, then $\tilde{\Sigma}=\{\tilde{\sigma}_{n}\colon G\to\Sym(\amalg_{a_{n}}V_{n})\}$ is a sofic approximation, where
	$$
	\tilde{\sigma}_{n}(g)w=\sigma_{n}(g)\kappa(w)\text{, for any }w\in \amalg_{a_{n}} V_{n}.
	$$
\end{lem}
\begin{proof}
	To check asymptotic freeness, choose some $\epsilon>0$ and $h,g\in G$, with  $h\neq g$. Take $N\in \mathbb{N}$ big enough such that for $k>N$ we have
	$$
	\frac{1}{|V_{k}|}|\{v\in V_{k}\colon \sigma_{k}(g)v\neq \sigma_{k}(h)v\}|\geq1-\epsilon.
	$$
	It is not hard to see that
	
	\[ 	|\{w\in \amalg_{a_{k}}V_{k}\colon \tilde{\sigma}_{k}(g)w\neq \tilde{\sigma}_{k}(h)w\}|=
	a_{n}|\{v\in V_{k}\colon \sigma_{k}(g)v\neq \sigma_{k}(h)v\}|\geq (1-\epsilon)a_{k}|V_{k}|, \]
	so dividing by $|\amalg_{a_{k}} V_{k}|=a_{k}|V_{k}|$ we obtain
	\[ 
	\frac{1}{|\amalg_{a_{k}} V_{k}|}|\{w\in \amalg_{a_{k}} V_{k}\colon \tilde{\sigma}_{k}(g)w\neq \tilde{\sigma}_{k}(h)w\}|\geq 1-\epsilon.
	\]Similarly, we prove asymptotic multiplicativeness.
\end{proof}
\begin{dfn}
	Let $U,V$ be finite sets with discrete metric such that $|U|=|V|$. Let $\sigma^{\text{\tiny U}}\colon G\to \Sym(U)$ and $\sigma^{\text{\tiny V}}\colon G\to \Sym(V)$ be some mappings. Fix $\delta>0$ and a finite subset $F\subset G$. We call a map $\phi\colon U\to V$ an $(F,\delta)$-isomorphism, if $d_{2}(\phi\sigma^{\text{\tiny U}}(f),\sigma^{\text{\tiny U}}(f)\phi)<\delta$ for every $f\in F$ and there are sets $\bar{U}\subset U$ and $\bar{V}\subset V$ with $\min\{|\bar{U}|/|U|,|\bar{V}|/|V|\}>1-\delta$ such that the restriction of $\phi$ to $\bar{U}$ is a bijection onto $\bar{V}$. We write $\bar{\phi}\colon \bar{U}\to\bar{V}$ for this map. The set of all $(F,\delta)$-isomorphisms $U\to V$ will be called $\iso(F,\delta,\sigma^{\text{\tiny U}},\sigma^{\text{\tiny V}})$.
\end{dfn}
\begin{dfn}\label{entropy-of-sofic-approximation}
	Let $\Sigma^{^{\text{\tiny V}}}=\{\sigma^{\text{\tiny V}}_{n}\colon G\to\Sym(V_{n})\}$ and $\Sigma^{\text{\tiny U}}=\{\sigma^{\text{\tiny U}}_{n}\colon G\to\Sym(U_{n})\}$ be sofic approximations with $|V_{n}|=|U_{n}|$ for every $n\in\nat$, we define the entropy of $\Sigma^{\text{\tiny V}}$ with respect to $\Sigma^{\text{\tiny U}}$ by
	$$
	h_{\Sigma^{\text{\tiny U}}}(\Sigma^{\text{\tiny V}})=\sup_{\epsilon>0}\inf_{F\subset G}\inf_{\delta>0}\limsup_{n\to\infty}\dfrac{\log N_{\epsilon}(\iso(F,\delta,\sigma_{n}^{\text{\tiny U}},\sigma_{n}^{\text{\tiny V}}),d_{2})}{|U_{n}|}.
	$$ We set $h_{\Sigma^{\text{\tiny U}}}(\Sigma^{\text{\tiny V}})=-\infty$ if there exists an $\epsilon>0$ such that for every finite $F\subset G$ and $\delta>0$ we have $\iso(F,\delta,\sigma_{n}^{\text{\tiny U}},\sigma_{n}^{\text{\tiny V}})=\emptyset$ for infinitely many $n\in\nat$.
\end{dfn}
For the sake of clarity, in the following proofs we will often abbreviate $U=U_{n}$, $V=V_{n}$ and $\sigma^{\text{\tiny V}}=\sigma^{\text{\tiny V}}_{n}$, $\sigma^{\text{\tiny U}}=\sigma^{\text{\tiny U}}_{n}$. Let $\Sigma^{\text{\tiny V}}$ and $\Sigma^{\text{\tiny U}}$ be sofic approximations, from now on we assume that $|U_{n}|=|V_{n}|$ for every $n\in\nat$.
\begin{lem}\label{lem-symmetry}
	For every sofic approximations $\Sigma^{\text{\tiny V}}=\{\sigma^{\text{\tiny V}}_{n}\colon G\to\Sym(V_{n})\}$ and $\Sigma^{\text{\tiny U}}=\{\sigma^{\text{\tiny U}}_{n}\colon G\to\Sym(U_{n})\}$, we have $$h_{\Sigma^{\text{\tiny U}}}(\Sigma^{\text{\tiny V}})=h_{\Sigma^{\text{\tiny V}}}(\Sigma^{\text{\tiny U}}).$$ Denote their common value by $h(\Sigma^{\text{\tiny U}},\Sigma^{\text{\tiny V}}).$
\end{lem}
\begin{proof}
	Assume that $\max\{h_{\Sigma^{\text{\tiny U}}}(\Sigma^{\text{\tiny V}}),h_{\Sigma^{\text{\tiny V}}}(\Sigma^{\text{\tiny U}})\}\geq 0$, since otherwise $h_{\Sigma^{\text{\tiny U}}}(\Sigma^{\text{\tiny V}})=-\infty$ and $h_{\Sigma^{\text{\tiny V}}}(\Sigma^{\text{\tiny U}})=-\infty$. In particular we can suppose $h_{\Sigma^{\text{\tiny U}}}(\Sigma^{\text{\tiny V}})\geq0$. Fix $\delta>0$, a finite set $F\subset G$ and $n\in\nat$. Let $\varphi\colon U\to V$ be an $(F\cup F^{-1},\delta)$-isomorphism. Let $n\in\nat$ be big enough so that $\sigma^{\text{\tiny V}}$ and $\sigma^{\text{\tiny U}}$ are $\delta/2$-multiplicative and $\delta/2$-free with respect to $F\cup F^{-1}$, in particular we have
\begin{align*}
	\bar{V}^{1}:=\{v\in V~|~\sigma^{\text{\tiny V}}_{g}\sigma^{\text{\tiny V}}_{g^{-1}}v=v\text{, for every }g\in F\cup F^{-1}\}\text{ with }& |\bar{V}^{1}|\geq(1-\delta)|V|\text{ and }\\
	\bar{U}^{1}:=\{u\in U~|~\sigma^{\text{\tiny U}}_{g}\sigma^{\text{\tiny U}}_{g^{-1}}u=u\text{, for every }g\in F\cup F^{-1}\}\text{ with }& |\bar{U}^{1}|\geq(1-\delta)|U|.
\end{align*}
	 Since $\varphi$ is an $(F\cup F^{-1},\delta)$-isomorphism, we can find sets $\bar{V}^{2}\subset V$ and $\bar{U}^{2}\subset U$ with $\min\{|\bar{U}^{2}|/|U|,|\bar{V}^{2}|/|V|\}\geq 1-\delta$ such that $\varphi|_{\bar{U}^{2}}\colon \bar{U}^{2}\to\bar{V}^{2}$ is a bijection. Therefore $|\bar{U}^{1}\cap \bar{U}^{2}|\geq (1-2\delta)|U|$ and $|\bar{V}^{1}\cap \bar{V}^{2}|\geq (1-2\delta)|V|$. Let 
	  \[ 
	  \bar{U}:=\bigcap_{g\in F\cup F^{-1}}\varphi^{-1}({(\sigma^{\text{\tiny V}}_{g})}^{-1}(\varphi(\bar{U}^{1}\cap\bar{U}^{2}))\cap\bar{V}^1\cap  \bar{V}^2)\cap \bar{U}^{1}\cap \bar{U}^{2}.
	  \]Therefore 
	  $$|\bar{U}|\geq (1-(8|F|+2)\delta)|U|\geq(1-10|F|\delta)|U|.$$ Note that $(\sigma^{\text{\tiny V}}_{g}\varphi)|_{\bar{U}}\colon \bar{U}\to T_{g}:=(\sigma^{\text{\tiny V}}_{g}\varphi)(\bar{U})$ is a bijection. Extend $\bar{\varphi}^{-1}$ in any way to $\psi\colon V\to U$. Since $d_{2}(\sigma^{\text{\tiny V}}_{g}\circ\varphi,\varphi\circ\sigma^{\text{\tiny U}}_{g})<\delta$, there must exist set $\ddot{U}\subseteq U$ with cardinality at least $(1-2\delta|F|)|U|$ such that $(\sigma^{\text{\tiny V}}_{g}\circ\varphi)u=(\varphi\circ\sigma^{\text{\tiny U}}_{g})u$ for every $u\in \ddot{U}$ and $g\in F\cup F^{-1}$. Since $\sigma^{\text{\tiny V}}_{g}\circ\varphi$ is a bijection from $\bar{U}\cap\ddot{U}$ to $T_{g}:=\sigma^{\text{\tiny V}}_{g}\circ\varphi(\bar{U}\cap\ddot{U})$ for every $g\in F\cup F^{-1}$, let us define $u_{v}\in \bar{U}\cap\ddot{U}$ so that $(\sigma^{\text{\tiny V}}_{g}\circ\varphi) u_{v}=v$ for $v\in T_{g}$ and compute
	\begin{align*}
	\sum_{v\in T_{g}}d(\psi\circ \sigma^{\text{\tiny V}}_{g^{-1}}(v),\sigma^{\text{\tiny U}}_{g^{-1}}\circ\psi(v))^{2}=&\\
	\sum_{u_{v}\in \bar{U}\cap \ddot{U}}d(\psi\circ \sigma^{\text{\tiny V}}_{g^{-1}}((\sigma^{\text{\tiny V}}_{g}\circ\varphi) u_{v}),\sigma^{\text{\tiny U}}_{g^{-1}}\circ\psi((\sigma^{\text{\tiny V}}_{g}\circ\varphi) u_{v}))^{2}=&\\
	\sum_{u_{v}\in \bar{U}\cap \ddot{U}}d((\psi\circ \varphi) u_{v},\sigma^{\text{\tiny U}}_{g^{-1}}\circ\psi((\sigma^{\text{\tiny V}}_{g}\circ\varphi) u_{v}))^{2}=&\\
	\sum_{u_{v}\in \bar{U}\cap \ddot{U}}d((\psi\circ \varphi) u_{v},(\sigma^{\text{\tiny U}}_{g^{-1}}\circ\psi\circ\varphi\circ\sigma^{\text{\tiny U}}_{g}) u_{v})^{2}=&\\
	\sum_{u_{v}\in \bar{U}\cap \ddot{U}}d(u_{v},u_{v})^{2}=&0.
	\end{align*}In the second and fourth equality, we have used the fact that $\sigma^{\text{\tiny V}}_{g^{-1}}\sigma^{\text{\tiny V}}_{g}v=v$ and $\sigma^{\text{\tiny U}}_{g^{-1}}\sigma^{\text{\tiny U}}_{g}u=u$ for every $v\in \bar{V}^{2}$ and $u\in \bar{U}$. Note also that $\varphi(u_{v})\in\bar{V}^{2}$ for every $v\in T_{g}$. In the fourth equality, we have used that $(\psi\circ \varphi)u=u$ for every $u\in \bar{U}$.
	Note that cardinality of $T_{g}$ satisfies $|T_{g}|\geq(1-(10|F|+2)\delta)|V|\geq(1-12|F|\delta)|V|$, hence
	$$
	\sum_{v\in V}d(\psi\circ \sigma^{\text{\tiny V}}_{g^{-1}}(v),\sigma^{\text{\tiny U}}_{g^{-1}}\circ\psi(v))^{2}\leq |V|-|T|\leq 12\delta|F||V|
	$$ and so $\psi$ is an $(F\cup F^{-1},\sqrt{ 12\delta|F|})$-isomorphism. Therefore to every $(F\cup F^{-1},\delta)$-isomorphism we can assign an $(F\cup F^{-1},\sqrt{ 12\delta|F|})$-isomorphism. We will prove that such assignment will map $(\epsilon,d_{2})$-separated sets to $(\epsilon/2,d_{2})$-separated sets.
	
	Let $\mathcal{M}$ be an $(\epsilon,d_{2})$-separated set in $\iso(F\cup F^{-1},\delta,\sigma^{\text{\tiny U}},\sigma^{\text{\tiny V}})$. Then for every distinct $\varphi_{1},\varphi_{2}\in\mathcal{M}$ we have $d_{2}(\varphi_{1},\varphi_{2})\geq\epsilon$, so that $\varphi_{1}(u)\neq\varphi_{2}(u)$ on the subset $\tilde{U}\subset U$ of  cardinality at least $\epsilon^2|U|$. Let $\bar{\varphi}_{1}:=\varphi|_{\bar{U}^{1}}\colon \bar{U}^{1}\to\bar{V}^{1}$ and $\bar{\varphi}_{2}:=\varphi|_{\bar{U}^{2}}\colon \bar{U}^{2}\to\bar{V}^{2}$ be bijections with $\min\{|\bar{U}^{1}|,|\bar{U}^{2}|\}\geq(1-\delta)|U|$. Extend $\bar{\varphi}_{1}^{-1}$ to $\psi_{1}\colon V\to U$ and $\bar{\varphi}_{2}^{-1}$ to $\psi_{2}\colon V\to U$. Note that $\psi_{1}$ and $\psi_{2}$ are bijections on $V'=\varphi_{1}(\tilde{U}\cap \bar{U}^{1}\cap\bar{U}^{2})\cap \varphi_{2}(\tilde{U}\cap \bar{U}^{1}\cap\bar{U}^{2})$. We claim that $\psi_{1}(v)\neq\psi_{2}(v)$ for every $v\in V'$. Indeed, find $u_{1}$ and $u_{2}$ in $\tilde{U}\cap \bar{U}^{1}\cap\bar{U}^{2}$ such that $\varphi_{1}(u_{1})=v=\varphi_{2}(u_{2})$, then $\psi_{1}(v)=\psi_{2}(v)$ leads to contradiction, since it implies that $u_{1}=u_{2}$, but $\varphi_{1}$ and $\varphi_{2}$ must differ on $\tilde{U}$. Note that we can estimate $|V'|\geq(\epsilon^2-4\delta)|V|$ and as a consequence $d_{2}(\psi_{1},\psi_{2})\geq\sqrt{\epsilon^2-4\delta}\geq\epsilon/2$, if $\delta$ is small enough.
	It yields the inequality on $d_{2}$-separated sets
	$$
	N_{\epsilon}(\iso(F\cup F^{-1},\delta,\sigma_{n}^{\text{\tiny U}},\sigma_{n}^{\text{\tiny V}}),d_{2})\leq N_{\epsilon/2}(\iso(F\cup F^{-1},\sqrt{ 12\delta|F|},\sigma_{n}^{\text{\tiny V}},\sigma_{n}^{\text{\tiny U}}),d_{2}).
	$$Taking the appropriate limits, the inequality above yields $h_{\Sigma^{\text{\tiny U}}}(\Sigma^{\text{\tiny V}})\leq h_{\Sigma^{\text{\tiny V}}}(\Sigma^{\text{\tiny U}})$.
	Note that, since $h_{\Sigma^{\text{\tiny U}}}(\Sigma^{\text{\tiny V}})\geq0$ there must be also $h_{\Sigma^{\text{\tiny V}}}(\Sigma^{\text{\tiny U}})\geq0$. We can now repeat the whole reasoning with $\Sigma^{\text{\tiny U}}$ interchanged with $\Sigma^{\text{\tiny V}}$. 
\end{proof}
\begin{lem}\label{lem-entropy-of-sofic-approximation}
	If $h(\Sigma^{\text{\tiny U}},\Sigma^{\text{\tiny V}})\geq0$, then for every dynamical system $(X,G)$ we have
	$$
	h_{\Sigma^{\text{\tiny V}}}(X,G)=h_{\Sigma^{\text{\tiny U}}}(X,G).
	$$
\end{lem}
\begin{proof}
	We may assume that at least one of $h_{\Sigma^{\text{\tiny V}}}(X,G)$, $h_{\Sigma^{\text{\tiny U}}}(X,G)$ is non negative, since otherwise equality is obvious. Let $h_{\Sigma^{\text{\tiny V}}}(X,G)\geq 0$. 
	We will prove that 
	\begin{equation}\label{claim}
	h_{\Sigma^{\text{\tiny V}}}(X,G)\leq h_{\Sigma^{\text{\tiny U}}}(X,G).
	\end{equation}Consequently, we will have $h_{\Sigma^{\text{\tiny U}}}(X,G)\geq 0$. In other words, if only one of the entropies $h_{\Sigma^{\text{\tiny V}}}(X,G)$, $h_{\Sigma^{\text{\tiny U}}}(X,G)$ is non negative, both of them must be. Then by interchanging sofic approximations $\Sigma^{\text{\tiny V}}$ and $\Sigma^{\text{\tiny V}}$ we will obtain equality in \eqref{claim}.
	
	Fix some $ \epsilon>0,\delta>0 $, a finite set $F\subset G$ and $n\in\nat$ big enough so that $ \iso(F,\delta,\sigma_{n}^{\text{\tiny U}},\sigma_{n}^{\text{\tiny V}})\neq\emptyset $. Put $\sigma^{\text{\tiny V}}=\sigma^{\text{\tiny V}}_{n}$, $\sigma^{\text{\tiny U}}=\sigma^{\text{\tiny U}}_{n}$. Take an $(F,\delta)$-isomorphism $\psi\colon U_{n}\to V_{n}$ and an $(F,\delta)$-pseudoorbit $\varphi\colon V_{n}\to X$. We will prove, that $\varphi\circ\psi\colon U_{n}\to X$ is  an $(F,\sqrt{8\delta})$-pseudoorbit. We have to check that $d_{2}(f\varphi\circ\psi, \varphi\circ\psi\circ\sigma^{\text{\tiny U}}_{f})^{2}<8\delta$ holds for every $f\in F$. Let $\bar{\psi}\colon\bar{U}\to\bar{V}$ be a bijection such that $\psi|_{\bar{U}}=\bar{\psi}$. Then
	\begin{multline}\label{sum}
	d_{2}(f\varphi\circ\psi, \varphi\circ\psi\circ\sigma^{\text{\tiny U}}_{f})^{2}=\\\frac{1}{|U_{n}|}\sum_{u\in \bar{U}}d(f\varphi\psi(u),\varphi\psi\sigma^{\text{\tiny U}}_{f}(u))^2+
	\frac{1}{|U_{n}|}\sum_{u\in U_{n}\setminus\bar{U}}d(f\varphi\psi(u),\varphi\psi\sigma^{\text{\tiny U}}_{f}(u))^2=\\
	\frac{1}{|U_{n}|}\sum_{v\in \bar{V}}d(f\varphi(v),\varphi\psi\sigma^{\text{\tiny U}}_{f}(\bar{\psi}^{-1}(v)))^2+
	\frac{1}{|U_{n}|}\sum_{u\in U_{n}\setminus\bar{U}}d(f\varphi\psi(u),\varphi\psi\sigma^{\text{\tiny U}}_{f}(u))^2,
	\end{multline}where we substituted $v=\psi(u)$. Now, confine our attention to the first sum. As $d_{2}(\sigma_{f}^{\text{\tiny V}}\psi,\psi\sigma_{f}^{\text{\tiny U}})<\delta$ there exists a set $\ddot{U}\subset U_{n}$ with cardinality at least $(1-\delta)|U_{n}|$ such that $\sigma_{f}^{\text{\tiny V}}(\psi(u))=\psi\sigma_{f}^{\text{\tiny U}}(u)$ for every $u\in \ddot{U}$. Therefore $\varphi\sigma_{f}^{\text{\tiny V}}(v)=\varphi\psi\sigma_{f}^{\text{\tiny U}}(\bar{\psi}^{-1}(v))$ for every $v\in\bar{\psi}(\ddot{U}\cap \bar{U})$. Note that cardinality of $\bar{\psi}(\ddot{U}\cap \bar{U})$ must not be smaller then $(1-2\delta)|V_{n}|=(1-2\delta)|U_{n}|$. Consequently, $\sum_{v\in \bar{V}}d(\varphi\sigma_{f}^{\text{\tiny V}}(v),\varphi\psi\sigma_{f}^{\text{\tiny U}}(\bar{\psi}^{-1}(v)))^2<2\delta|U_{n}|$. Let us compute
	\begin{multline*}
	\frac{1}{|U_{n}|}\sum_{v\in \bar{V}}d(f\varphi(v),\varphi\psi\sigma^{\text{\tiny U}}_{f}(\bar{\psi}^{-1}(v)))^2\leq\\
	\frac{1}{|U_{n}|}\sum_{v\in \bar{V}}\big(d(f\varphi(v),\varphi\sigma_{f}^{\text{\tiny V}}(v))+d(\varphi\sigma_{f}^{\text{\tiny V}}(v),\varphi\psi\sigma_{f}^{\text{\tiny U}}(\bar{\psi}^{-1}(v)))\big)^2\leq\\
	\frac{1}{|U_{n}|}\sum_{v\in \bar{V}}d(f\varphi(v),\varphi\sigma_{f}^{\text{\tiny V}}(v))^2+\frac{1}{|U_{n}|}\sum_{v\in \bar{V}}d(\varphi\sigma_{f}^{\text{\tiny V}}(v),\varphi\psi\sigma_{f}^{\text{\tiny U}}(\bar{\psi}^{-1}(v)))^2+\\
	\frac{2}{|U_{n}|}\sum_{v\in \bar{V}}d(\varphi\sigma_{f}^{\text{\tiny V}}(v),\varphi\psi\sigma_{f}^{\text{\tiny U}}(\bar{\psi}^{-1}(v)))d(f\varphi(v),\varphi\sigma_{f}^{\text{\tiny V}}(v))\leq
	\delta^2+ 2\delta+4\delta<7\delta,
	\end{multline*}since $d_{2}(f\varphi,\varphi\sigma_{f}^{\text{\tiny V}})^{2}\leq \delta^{2}$.
	On the other side, we find that the second sum in \eqref{sum} can be bounded
	$$
	\frac{1}{|U_{n}|}\sum_{u\in U_{n}\setminus\bar{U}}d(f\varphi\psi(u),\varphi\psi\sigma^{\text{\tiny U}}_{f}(u))^2\leq\delta,
	$$since  $\frac{|U_{n}\setminus\bar{U}|}{|U_{n}|}\leq\delta$. Therefore 
	$$
	d_{2}(f\varphi\circ\psi, \varphi\circ\psi\circ\sigma^{\text{\tiny U}}_{f})^{2}\leq 8\delta
	$$ and $\varphi\circ\psi$ is an $(F,\sqrt{8\delta})$-pseudoorbit.
	
	We will now prove that the function 
	$$\map(d,F,\delta,\sigma^{\text{\tiny V}})\ni\varphi\mapsto\Psi(\varphi):=\varphi\circ\psi\in \map(d,F,\sqrt{8\delta},\sigma^{\text{\tiny U}})
	$$ maps $(\epsilon,d_{2})$-separated sets to $(\sqrt{\epsilon},d_{\infty})$-separated sets, in particular we will show that $\Psi$ is injective on $(\epsilon,d_{2})$-separated sets for any $\epsilon>0$.
	Take $(F,\delta)$-pseudoorbits $\varphi_{1}$ and $\varphi_{2}$ with $d_{2}(\varphi_{1},\varphi_{2})\geq\epsilon$. It follows that $d(\varphi_{1}(v),\varphi_{2}(v))\geq\sqrt{\epsilon}$ on the subset $V'$ of $V_{n}$ with cardinality at least $\epsilon|V_{n}|$. Additionally, we can find $U'\subset U_{n}$ with cardinality at least $(1-\delta)|U_{n}|=(1-\delta)|V_{n}|$ such that $\psi|_{U'}\colon U'\to\psi(U')$ is a bijection. Hence if $\delta$ is small enough then $U''=U'\cap \psi^{-1}(V'\cap \psi(U'))$ is nonempty, since $|U''|\geq (\epsilon-2\delta)|U_{n}|$ and there exists $u\in U''$ such that $d(\varphi_{1}(\psi(u)),\varphi_{2}(\psi(u)))\geq\sqrt{\epsilon}$ and so $d_{\infty}(\varphi_{1}\circ\psi,\varphi_{2}\circ\psi)\geq\sqrt{\epsilon}$. It proves our claim.
	
	We have showed that
	$$
	N_{\epsilon}(\map(d,F,\delta,\sigma^{\text{\tiny V}}),d_{2})\leq N_{\sqrt{\epsilon}}(\map(d,F,\sqrt{8\delta},\sigma^{\text{\tiny U}}),d_{\infty}).
	$$Taking the logarithms and dividing by $|U_{n}|$ we get
	
	\[ \frac{\log N_{\epsilon}(\map(d,F,\delta,\sigma^{\text{\tiny V}}),d_{2})}{|V_{n}|}\leq\frac{\log N_{\sqrt{\epsilon}}(\map(d,F,\sqrt{8\delta},\sigma^{\text{\tiny U}}),d_{\infty})}{|U_{n}|}, \]
	 and finally applying appropriate limits and using Remark \ref{d2-dinfty} we obtain claim \eqref{claim}.
\end{proof}
\begin{rem}
	It is worth noting that entropy of a dynamical system $(X,G)$ depends on a sofic sequence. First example of such phenomenon was presented by Lewis Bowen in \cite[Thm. 4.1.]{Bowen2019}. It is a system with infinite negative entropy and zero entropy, with respect to different sofic sequences. It was by no means satisfactory, since it left open major problem, namely, whether there exists a system with two different positive entropies. It was finally answered positively in late 2019 by Dylan Airey, Lewis Bowen and Frank Lin in \cite{DBF2019}.	
	%By lemma (\ref{lem-entropy-of-sofic-approximation}) we deduce that possible values of $h(\Sigma^{\text{\tiny U}},\Sigma^{\text{\tiny V}})$ are $0$ and $-\infty$, since inequality in (\ref{lem-entropy-of-sofic-approximation}) holds for any system with finite entropy and there always exists such systems. Moreover, if $h(\Sigma^{\text{\tiny U}},\Sigma^{\text{\tiny V}})=0$ then $h_{\Sigma^{\text{\tiny V}}}(X,G)= h_{\Sigma^{\text{\tiny U}}}(X,G)$ for any system $(X,G)$. In particular they must be simultaneously positive or negative infinite. Additionally, if there is a system with negative and non-negative entropies, then $h(\Sigma^{\text{\tiny U}},\Sigma^{\text{\tiny V}})=-\infty$. Note that it is not true, that if $h(\Sigma^{\text{\tiny U}},\Sigma^{\text{\tiny V}})=-\infty$ then entropies of any system will differ. We can produce an example using full shift, for which entropy is always equal to logarithm of cardinality of the alphabet, and sofic approximations discovered by Bowen, for which there is system of entropies $0$ and $-\infty$.
\end{rem}
We now prove two key lemmas that will help us to formulate results about entropy of Toeplitz systems independently or partially independently of a sofic approximation.
\begin{thm}\label{characterisation}
	Let $\Sigma=\{\sigma_{n}\colon G\to\Sym(V_{n})\}$  be a sofic sequence by homomorphisms with $K_{n}=\ker \sigma_{n}$. Then there exists a sequence $(f_{n})_{n\in\nat}\subset\nat$ such that $\ddot{\Sigma}=\{ \ddot{\sigma}_{n}\colon G\to\Sym(\amalg _{f_{n}}G/K_{n}) \}$, where $\ddot{\sigma}_{n}$ is obtained from $\sigma_{n}$ as in Lemma \ref{lem-amalg}, is a sofic sequence, and for every system $(X,G)$ we have $h_{\Sigma}(X,G)\leq h_{\ddot{\Sigma}}(X,G)$.
\end{thm}
\begin{proof}
	Let us abbreviate notation of the action $G$ on $V_{n}$ by $\sigma_{n}(g)v=gv$ for every $n\in\nat$, $v\in V_{n}$ and $g\in G$. Define $K_{n}:=\Ker \sigma_{n}$, note that if $h\in \bigcap_{n\in\mathbb{N}}K_{n}$ then $hv=v$ for every $v\in V_{n}$ which contradicts asymptotic freeness, unless $h=e$. Hence $\bigcap_{n\in\mathbb{N}}K_{n}=\{e\}$ and each $K_{n}$ is finitely indexed and normal, in particular $G$ is residually finite.
	
	Let $n\in \nat$. Notice that for every $v\in V_{n}$ the mapping	$$\Psi_{v}\colon G\ni g\mapsto gv\in V_{n}$$
	is a homomorphism of $G$-sets and its set of fixed points is the stabilizer of $v$ for the $\sigma_{n}$ action on $V_{n}$. We will denote this stabilizer by $\stab(v)=\{g\in G\colon gv=v\}$. By elementary algebra it induces an isomorphism
	$$
	\tilde{\Psi}_{v}\colon G/\stab(v) \ni g\stab(v)\mapsto gv\in \im\Psi_{v},
	$$
	where $\im\Psi_{v}$ is an orbit of $v$ by the action of $G$ on $V_{n}$. Given distinct $v,w\in V_{n}$ we have either $\{gv\}_{g\in G}\cap\{gw\}_{g\in G}=\emptyset$ or  $h_{1}v=h_{2}w$ for some $h_{1},h_{2}\in G$. In the latter case  $h_{2}^{-1}h_{1}v=w$ and 
	$$\{gv\}_{g\in G}=\{gh_{2}^{-1}h_{1}v\}_{g\in G}=\{gw\}_{g\in G},$$ where in the first equality we have used that transformation $g\mapsto gh_{2}^{-1}h_{1}$ is a bijection on $G$. It follows that the images of $\Psi_{v}$ and $\Psi_{w}$ are either disjoint or equal. Hence for some $f_{n}\in\nat$ and $v_{1},...,v_{f_{n}}\in V_{n}$ the mapping
	$$
	(\Psi_{v_{i}})_{i=1}^{f_{n}}\colon \amalg_{f_{n}}G\ni (j, g)\mapsto (j,gv_{j})\in \amalg_{i=1}^{f_{n}}\im\Psi_{v_{i}}=V_{n}
	$$
	is a factor map of $G$-sets and as before it induces an isomorphism
	$$
	(\tilde{\Psi}_{v_{i}})_{i=1}^{f_{n}}\colon \amalg_{i=1}^{f_{n}}G/\stab(v_{i})\ni (j, g\stab(v_{j}))\mapsto (j,gv_{j})\in \amalg_{i=1}^{f_{n}}\im\Psi_{v_{i}}=V_{n}.
	$$	
	Define the sofic sequence $\tilde{\Sigma}=\{ \tilde{\sigma}_{n}\colon G\to\Sym(\amalg_{i=1}^{f_{n}}G/H_{i}^{n}) \}$, where $H_{i}^{n}=\stab(v_{i})$,  by the formula $\ddot{\sigma}_{n}(g)(i,cG/H_{i}^{n})=(i,cg^{-1}H_{i}^{n})$ for every $i\leq f_{n}$ and $c,g\in G$. Since $(\tilde{\Psi}_{v_{i}})_{i=1}^{f_{n}}$ is an isomorphism of $G$-sets $(\amalg_{i=1}^{f_{n}}G/H_{i}^{n},\tilde{\sigma}_{n})$ and $(V_{n},\sigma_{n})$, it is an $(F,\delta)$-isomorphism for every finite $F\subset G$ and $\delta>0$. Therefore $h(\Sigma,\tilde{\Sigma})\geq0$ and Theorem \ref{lem-entropy-of-sofic-approximation} implies that $h_{\Sigma}(X,G)=h_{\tilde{\Sigma}}(X,G)$.
	
	Finally put $\ddot{\Sigma}=\{ \ddot{\sigma}_{n}\colon G\to\Sym(\amalg_{i=1}^{f_{n}}G/K_{n}) \}$, where $\ddot{\sigma}_{n}(g)(i,cK_{n})=(i,cg^{-1}K_{n})$ for every $i\leq f_{n}$ and $c,g\in G$. By Lemma \ref{lem-amalg} we know that $\ddot{\Sigma}$ is a sofic approximation.
	
	Let $n\in\nat$. Note that if $\phi\colon \amalg_{i=1}^{f_{n}}G/H_{i}^{n} \to X$ is an $(F,\delta)$-pseudoorbit, for some finite $F\subset G$ and $\delta>0$, then $\Psi(\phi)=\psi\colon\amalg_{i=1}^{f_{n}}G/K_{n}\to X$ with $\psi(i,gK_{n})=\phi(i,gH_{i}^{n})$, for every $g\in G$ and $i\leq f_{n}$, is well defined, and is an $(F,\delta)$-pseudoorbit. Indeed, it is enough to note that $K_{n}\subset H_{i}^{n}$ and for every $g\in G$, $h\in F$ and $i\leq f_{n}$ we have
\begin{multline*}
	d(h\psi(i,gK_{n}),\psi\ddot{\sigma}_{h}(i,gK_{n}))=d(h\psi(i,gK_{n}),\psi(i,gh^{-1}K_{n}))=\\
	d(h\phi(i,gH_{i}^{n}),\phi(i,gh^{-1}H_{i}^{n}))=d(h\phi(i,gH_{i}^{n}),\phi\tilde{\sigma}_{h}(i,gH_{i}^{n})),
\end{multline*}
	so $d_{2}(h\psi,\psi\ddot{\sigma}_{h})=d_{2}(h\phi,\phi\tilde{\sigma}_{h})$, since $g\in G$ was arbitrary.
	We can now define the function 
	$$\map(d,F,\delta,\tilde{\sigma}_{n})\ni\phi \mapsto \Psi(\phi)\in\map(d,F,\delta,\ddot{\sigma}_{n}).$$ Let us check that if $E\subset \map(d,F,\delta,\tilde{\sigma}_{n})$ is an $(\epsilon, d_{\infty})$-separated set then $\Psi(E)$ is also an $(\epsilon, d_{\infty})$-separated set. This is because if $\phi$ and $\phi'$ are in $E$ and satisfy $d_{\infty}(\phi,\phi')\geq\epsilon$ then  $d(\phi(i_{0},g_{0}H_{i_{0}}^{n}),\phi'(i_{0},g_{0}H_{i_{0}}^{n}))\geq\epsilon$ for some $g_{0}\in G$ and $i_{0}\leq f_{n}$. Note that $$d(\phi(i_{0},g_{0}H_{i_{0}}^{n}),\phi'(i_{0},g_{0}H_{i_{0}}^{n}))=d(\Psi(\phi)(i_{0},g_{0}K_{n}),\Psi(\phi')(i_{0},g_{0}K_{n})).$$ Therefore $\Psi(E)$ is also an $(\epsilon,d_{\infty})$-separated set, so $h_{\tilde{\Sigma}}(X,G)\leq h_{\ddot{\Sigma}}(X,G)$ and the main claim follows.
\end{proof}
\begin{dfn}\label{1}
	Let $\Sigma'=\{\sigma'_{n}\colon G\to\Sym(V_{n})\}_{n\in\nat}$ be a sofic approximation to a free group $G=\langle S  \rangle$ generated by a finite set $S$. We define a sequence $\Sigma=\{\sigma_{n}\colon G\to\Sym(V_{n})\}_{n\in\nat}$ by the formula \[ \sigma_{n}(f)v:=(\sigma'_{n}(s_{1})^{\alpha_{1}}\circ\ldots\circ\sigma'_{n}(s_{k})^{\alpha_{k}})v,
	\] for any reduced element $ f=\prod_{i=1}^{k}s_{i}^{\alpha_{i}}\in G $, where $ s_{i}\in S $  and $ \alpha_{i}\in\nat $ for $ i=1,\ldots k $, and some $k\in \nat$.
\end{dfn}
\begin{lem}\label{free-group-sofic-app}
		If $\Sigma'=\{\sigma'_{n}\colon G\to\Sym(V_{n})\}_{n\in\nat}$ is a sofic approximation to a free group $G=\langle S  \rangle$ generated by a finite set $S$, then $\Sigma=\{\sigma_{n}\}_{n\in\nat}$ defined according to Definition \ref{1} is a sofic approximation by homomorphisms equivalent to $\Sigma'$.
\end{lem}
\begin{proof}
	Since $S$ is a set of generators, every $\sigma_{n}$ is indeed a well defined homomorphism. We will prove that for every finite subset $F\subset G$ and $\delta>0$ there is $N\in \nat$ such that for every $n>N$ the identity map $\operatorname{id}\colon V_{n}\to V_{n}$ is an $(F,\delta)$-isomorphism. Note that soficity of $\Sigma'$ implies that there exists $N\in\nat$ such that for every $n>N$ we have $|U|<\delta^{2}|V_{n}|$, where
	\begin{multline*}
	U=\{v\in V_{n}~|~(\sigma'_{n}(s_{1})^{\alpha_{1}}\circ\ldots\circ\sigma'_{n}(s_{k})^{\alpha_{k}})v\neq\sigma'_{n}(f)v,\\\text{ where } f=\prod_{i=1}^{k}s_{i}^{\alpha_{i}}\in F\text{ is a reduced word in }\langle S\rangle,\text{for some }k\in\nat\}.
	\end{multline*}
	Since $\sigma_{n}$ is defined by the action on generators, the cardinality of the set of all elements $v$ in $V_{n}$ on which $d(\sigma_{n}(f)v,\sigma'_{n}(f)v)=1$ for some $f\in F$ is equal to the cardinality of $U$. Therefore for every $f\in F$ there is $d_{2}(\sigma_{n}(f),\sigma'_{n}(f))<\delta$ and $\operatorname{id}$ is an $(F,\delta)$-isomorphism. We finish the proof applying Lemma \ref{lem-entropy-of-sofic-approximation}.
\end{proof}
Combining Theorem \ref{characterisation} and Lemma \ref{free-group-sofic-app} we obtain
\begin{thm}\label{free group characterisation}
	Let $G$ be free group generated by a finite set. Then for every sofic sequence $\Sigma=\{\sigma_{n}\colon G\to\Sym(V_{n})\}$ there exists a sequence $K_{n}\searrow\{e\}$ of finitely indexed normal subgroups of $G$ and a sequence $(f_{n})_{n\in\nat}\subset\nat$ such that $\ddot{\Sigma}=\{ \ddot{\sigma}_{n}\colon G\to\Sym(\amalg _{f_{n}}G/K_{n}) \}$ is a sofic approximation and for every system $(X,G)$ we have $h_{\Sigma}(X,G)\leq h_{\ddot{\Sigma}}(X,G)$.
\end{thm}
\section{Entropy of subshifts over residually finite group}\label{Entropy of subshifts over residually finite group}
First, we need to introduce another definition of the sofic topological entropy, see \cite{Bowen2017} and \cite{Austin2016}. Assume that $\Sigma=\{\sigma_{n}\colon G \to V_{n}\}$ is a sofic approximation sequence for a sofic group $G$.

Let $X \subset k^{G}$ be a $G$-shift. Fix a function $\phi\colon V_{n}\to k$ and an element $v\in V_{n}$. Define \emph{the pullback} $\Pi_{v}^{\sigma_{n}}(\phi)\in k^{G}$ of $\phi$ by the formula $\Pi_{v}^{\sigma_{n}}(\phi)(g)=\phi(\sigma_{n}(g)^{-1}v)$. Given an open neighborhood $\mathcal{U}$ of $X$ in $k^{G}$ and $\delta>0$, let $\Omega(\sigma_{n}, \delta, \mathcal{U})$ be the set of all maps $\phi\colon V_{n}\to k$ such that
$$
|V_{n}|^{-1}|\{v\in V_{n}\colon \Pi_{v}^{\sigma_{n}}(\phi)\in \mathcal{U}\}| \geqslant 1-\delta.
$$
We call such a map a $(\sigma_{n},\delta,\mathcal{U})$-\emph{microstate}, or $(\delta,\mathcal{U})$-\emph{microstate} when no confusion can arise. Define the sofic topological entropy of $(X,G)$ by
$$
\tilde{h}_{\Sigma}(X,G)=\inf_{\delta>0}\inf_{\mathcal{U}\supset X}\limsup_{n\to\infty}|V_{n}|^{-1}\log |\Omega(\sigma_{n}, \delta, \mathcal{U})|.
$$
It turns out that this definition of the sofic entropy coincides with Definition \ref{entropy-def1}, proof can be found in \cite{Austin2016}, although it involves application of the sofic measure entropy and variational principle. Some intuitions on why those entropies should be equal were given by Bowen in \cite{Bowen2017}, but to the best of our knowledge formal direct proof never appeared in literature.
\begin{rem}
	In the case of a symbolic action, that is if $X$ is a $G$-shift, we will always use the following pseudometric
		\[
	d(x,y) =
	\left\{ \begin{array}{ll}
	1 & \textrm{ if }x_{e}\neq y_{e} ,\\
	0 & \textrm{ if }x_{e} = y_{e},
	\end{array} \right.
	\]where $x,y\in X$. It is easy to check that $d$ is indeed a continuous dynamically generating pseudometric on any $G$-shift.
\end{rem}
\begin{dfn}
	For a finite subset $F\subset G$ define $\proj_{F}\colon k^{G}\to k^{F}$ to be the projection $k^{G}\ni x\mapsto x|_{F}\in k^{F}$. From now on we will denote $\Uee_{F}:=\proj_{F}^{-1}(\proj_{F}(X))$ for any finite $F\subset G$. Moreover, for any $w\in k^{F}$ we define a cylinder $[w]$ based on $w$ by the formula	$ [w]= \proj_{F}^{-1}(w)$. Note that $\mathcal{U}_{F}=\bigcup_{w \in \mathcal{B}_{F}(X)}[w].$
\end{dfn}
\begin{lem}\label{U_F}
	For every $G-$shift $(X,G)$ over $k\in\nat$ and sofic approximation $\Sigma$ we have
	\[ 
	\tilde{h}_{\Sigma}(X,G)=\inf_{\delta>0}\inf_{F\subset G}\limsup_{n\to\infty}|V_{n}|^{-1}\log |\Omega(\sigma_{n}, \delta, \mathcal{U}_{F})|,
	 \]where $\Uee_{F}:=\proj_{F}^{-1}(\proj_{F}(X)).$
\end{lem}
\begin{proof}
	Let $\delta>0$ and $\Uee$ be an open neighborhood of $X$ in $k^{G}$. By compactness of $X$ there exists $F\subset G$ such that $\Uee_{F}:=\proj_{F}^{-1}(\proj_{F}(X))\subset \Uee$. It implies that\[ 
	 \Omega(\sigma_{n}, \delta, \mathcal{U})\supset\Omega(\sigma_{n}, \delta, \mathcal{U}_{F})
	 \]and consequently\[ 
	 \tilde{h}_{\Sigma}(X,G)\geq\inf_{\delta>0}\inf_{F\subset G}\limsup_{n\to\infty}|V_{n}|^{-1}\log |\Omega(\sigma_{n}, \delta, \mathcal{U}_{F})|.
	  \]Since the reverse inequality is clearly true, we have proved the lemma.
\end{proof}

\begin{lem}\label{sofic-entropy-symbolic-action}
	Let $X\subset k^{G}$ be a $G$-shift. Then for any sofic approximation $\Sigma$ of $G$ there is
	$$
	\tilde{h}_{\Sigma}(X,G)=h_{\Sigma}(X,G).
	$$
\end{lem}
\begin{proof}
	Let $\epsilon>0$. We will prove that $\tilde{h}_{\Sigma}(X,G)\geq h_{\Sigma}(X,G)$. Fix a finite set $F\subset G$, $\delta>0$ and $n\in\nat$ big enough so that the set
	\[ 
	\bar{V}:=\{v\in V_{n}~|~\sigma_{n}(g)^{-1}(v)=\sigma_{n}(g^{-1})v\text{, for every }g\in F\cup F^{-1}\}\text{ satisfies }|\bar{V}|\geq(1-\delta)|V_{n}|.
	 \]
	Let $\Uee_{F}:=\proj_{F}^{-1}(\proj_{F}(X))$ and let $\phi\colon V_{n}\to X$ be an $(F\cup F^{-1},\delta)-$pseudoorbit. We will show that $\psi=\Psi(\phi):=(\phi(v)_{e})_{v\in V_{n}}$ is a $(3\delta|F|,\Uee_{F})$-microstate. Let $g\in F\cup F^{-1}$. Since $d_{2}(g\phi,\phi\sigma_{g})<\delta$, there is $\ddot{V}_{g}\subset V_{n}$ such that $ \phi(\sigma_{n}(g)v)_{e}=(g\phi(v))_{e} $ for every $v\in\ddot{V}_{g}$ and $(1-\delta^2)|V_{n}|\geq |\ddot{V}_{g}|$. Therefore for every $v\in \bar{V}\cap \ddot{V}_{g}$ holds $\phi(\sigma_{n}(g)^{-1}v)_{e}=(g^{-1}\phi(v))_{e}$. We conclude that for every $v\in V':=\bar{V}\cap\bigcap_{g\in F\cup F^{-1}} \ddot{V}_{g}$ and for every $g\in F\cup F^{-1}$ we have $\phi(\sigma_{n}(g)^{-1}v)_{e}=(g^{-1}\phi(v))_{e}$. Note that we can estimate $|V'|\geq(1-3\delta|F|)|V_{n}|$. Fix $g\in F$, then for every $v\in V'$ we can write $$\Pi_{v}^{\sigma_{n}}(\psi)(g)=\phi(\sigma_{n}(g)^{-1}v)_{e}=(g^{-1}\phi(v))_{e}=\phi(v)_{g}.$$ Hence, we have $\Pi_{v}^{\sigma_{n}}(\psi)(g)\in \Uee_{F}$ on the set of at least cardinality $(1-3\delta|F|)|V_{n}|$ and $\psi$ is a $(3\delta|F|,\Uee_{F})$-microstate. Moreover, 
	$$
	\map(d,F\cup F^{-1},\delta,\sigma_{n})\ni\phi \mapsto \Psi(\phi)\in \Omega(\sigma_{n}, 3\delta|F|, \Uee_{F})
	$$ is injective on $(\epsilon, d_{\infty})$-separated sets for every $\epsilon>0$. Indeed, if $d_{\infty}(\phi_{1},\phi_{2})>0$ for $\phi_{1},\phi_{2}\in \map(d,F\cup F^{-1},\delta,\sigma_{n})$, then there exists $v\in V_{n}$ such that $\phi_{1}(v)_{e}\neq\phi_{2}(v)_{e}$, hence $$d_{\infty}(\Psi(\phi_{1}),\Psi(\phi_{2}))\geq d(\phi_{1}(v),\phi_{2}(v))>0.$$ Let $E\subset \map(d,F\cup F^{-1},\delta,\sigma_{n})$ be an $(\epsilon, d_{\infty})$-separated set of maximal cardinality. We have
	\[ 
	N_{\epsilon}(\map(d,F\cup F^{-1},\delta,\sigma_{n}),d_{\infty})=|E|\leq|\Psi(E)|\leq |\Omega(\sigma_{n}, 3\delta|F|, \mathcal{U}_{F})|\leq|\Omega(\sigma_{n}, 3\delta|F|, \mathcal{U}_{F})|.
	\] Let us take exponential growth of both sides, and then let $\delta\to0$ to learn that	 
\begin{multline*}
	\inf_{\delta>0}\limsup_{n\to\infty}\frac{1}{|V_{n}|}\log N_{\epsilon}(\map(d,F\cup F^{-1},\delta,\sigma_{n}),d_{\infty})\leq\\ \inf_{\delta>0}\limsup_{n\to\infty}\frac{1}{|V_{n}|}\log|\Omega(\sigma_{n}, 3\delta|F|, \mathcal{U}_{F})|\leq
	\limsup_{n\to\infty}\frac{1}{|V_{n}|}\log|\Omega(\sigma_{n}, \delta', \mathcal{U}_{F})|,
\end{multline*}
	holds for every $\delta'>0$. Now, taking infimum over all finite $F\subset G$ and letting $\delta'\to 0$, we obtain $h_{\Sigma}(X,G)\leq \tilde{h}_{\Sigma}(X,G)$.
	
	On the other hand, let $F\subset G$ be a finite set with $e\in F$ and $F'=F\cup F^{-1}$. Let $\psi\colon V_{n}\to k$ be a $(\delta,\Uee_{F'})$-microstate, where $\delta>0$. Find $n\in\nat$ big enough so that there exists $\bar{V}\subset V_{n}$ with $|\bar{V}|\geq(1-\delta) |V_{n}|$ such that $\sigma_{e}^{-1}=\id$ on $\bigcup_{g\in F'}\sigma_{g}(\bar{V})$ and $\sigma_{g^{-1}}^{-1}|_{\bar{V}}=\sigma_{g}|_{\bar{V}}$ for any $g\in F'$. If $\proj_{F'}^{-1}(\Pi_{v}^{\sigma_{n}}(\psi)|_{F'})\cap X\neq\emptyset$, define $\phi=\Phi(\psi)\colon V_{n}\to X$ in any way, so that condition $\phi(v)\in\proj_{F'}^{-1}(\Pi_{v}^{\sigma_{n}}(\psi)|_{F'})\cap X$ is satisfied and in case $\proj_{F'}^{-1}(\Pi_{v}^{\sigma_{n}}(\psi)|_{F'})\cap X=\emptyset$ choose $\phi(v)\in X$ so that $\phi(v)_{e}=\psi(v)$. We check that $\phi$ is an $(F',\sqrt{4|F|\delta})$-pseudoorbit. Fix $v\in \bar{V}$, $g\in F'$. Note that if for some $v\in V_{n}$ we have $\proj_{F'}^{-1}(\Pi_{v}^{\sigma_{n}}(\psi)|_{F'})\cap X=\emptyset$, then $\Pi_{v}^{\sigma_{n}}(\psi)\notin \Uee_{F'}$, so it can happen only on the set $\ddot{V}\subset V_{n}$ of cardinality at most $\delta|V_{n}|$. Put $V':=\bar{V}\setminus (\bigcap_{g\in F'}\sigma_{g}^{-1}(\ddot{V}))$. We have $|V'|\geq (1-4|F|\delta)|V_{n}|$. Assume that $v\in V'$ and compute	
	\begin{multline*}
	(\phi(\sigma_{g}(v)))_{e}=\Pi_{\sigma_{g}(v)}^{\sigma_{n}}(\psi)(e)=\psi(\sigma_{e}^{-1}(\sigma_{g}(v)))=\\\psi((\sigma_{g}(v)))=\psi((\sigma_{g^{-1}}^{-1}(v)))=
	\Pi_{v}^{\sigma_{n}}(\psi)(g^{-1})=\phi(v)_{g^{-1}}=(g\phi(v))_{e},
	\end{multline*}where in the first equality and the second to last equality we have used that $\phi(w)\in\proj_{F'}^{-1}(\Pi_{w}^{\sigma_{n}}(\psi)|_{F'})$ for every $w\in V'$.
	Therefore	
	 \begin{multline*}
		d_{2}(\phi\sigma_{g},g\phi)^2=\dfrac{1}{|V_{n}|}\sum_{v\in V_{n}}d(\phi(\sigma_{g}(v)),g\phi(v))^2=\\\frac{1}{|V_{n}|}\sum_{v\in V_{n}\setminus V'}d(\phi(\sigma_{g}(v)),g\phi(v))^2\leq \frac{|V_{n}\setminus V'|}{|V_{n}|}\leq 4|F|\delta
	\end{multline*} and $\phi$ is an $(F', \sqrt{4|F|\delta})$-pseudoorbit. As before, the mapping
	 \[
	  \Omega(\sigma_{n},\delta,\Uee_{F'})\ni \psi\mapsto\Phi(\psi)\in \map(d,F',\sqrt{4|F|\delta},\sigma_{n})
	  \]
	 has $(\epsilon, d_{\infty})$-separated image for every $\epsilon>0$, because $\psi(v)_{e}=\Phi(\psi)(v)_{e}$ for every $v\in V_{n}$. We can now estimate
	\[ 
	|\Omega(\sigma_{n},\delta,\Uee_{F'})|\leq N_{\epsilon}(\map(d,F',\sqrt{4|F|\delta},\sigma_{n}),d_{\infty}).
	 \]Let us apply exponential growth to get
\begin{align*}
	 \inf_{L\subset G}\limsup_{n\to\infty}\frac{1}{|V_{n}|}\log|\Omega(\sigma_{n},\delta,\Uee_{L})|&
	  \leq\limsup_{n\to\infty}\frac{1}{|V_{n}|}\log|\Omega(\sigma_{n},\delta,\Uee_{F'})|\\
	 &\leq\limsup_{n\to\infty}\frac{1}{|V_{n}|}\log N_{\epsilon}(\map(d,F',\sqrt{4|F|\delta},\sigma_{n}),d_{\infty})\\
	 &\leq\limsup_{n\to\infty}\frac{1}{|V_{n}|}\log N_{\epsilon}(\map(d,F,\sqrt{4|F|\delta},\sigma_{n}),d_{\infty}),
\end{align*}In the last inequality we have used that $F\subset F'$ implies $\map(d,F',\sqrt{4|F|\delta},\sigma_{n})\subset \map(d,F,\sqrt{4|F|\delta},\sigma_{n})$. Applying the appropriate limits, we obtain
\begin{multline*}
\tilde{h}_{\Sigma}(X,G)=\inf_{\delta>0}\inf_{L\subset G}\limsup_{n\to\infty}\frac{1}{|V_{n}|}\log|\Omega(\sigma_{n},\delta,\Uee_{L})|\leq\\
\sup_{\epsilon>0}\inf_{F\subset G}\inf_{\delta>0}\limsup_{n\to\infty}\frac{1}{|V_{n}|}\log N_{\epsilon}(\map(d,F,\sqrt{4|F|\delta},\sigma_{n}),d_{\infty})=h_{\Sigma}(X,G)
\end{multline*}
and the main claim follows.	 
\end{proof}
\begin{lem}\label{lemat-n!}
	For every $n\in \nat$ we have
	$$
	e\Big(\frac{n}{e}\Big)^{n}\leq n!\leq e n \Big(\frac{n}{e}\Big)^{n}.
	$$
\end{lem}
\begin{lem}\label{sofic-residually}
	Let $G$ be a residually finite group with a sofic approximation $\Sigma$ by homomorphisms, and $X\subset k^{G}$ be a $G$-shift. Suppose $\mathcal{F}=\{F_{n}\}_{n\in\mathbb{N}}$ is a telescoping sequence of fundamental domains of $H_{n}:=\ker\sigma_{n}\searrow\{e\}$ with the property that $\bigcup\mathcal{F}=G$. Then the sofic entropy of $X$ with respect to $\Sigma$ satisfies
	$$
	h_{\Sigma}(X,G)\leq\liminf_{n\to\infty}\frac{\log |\mathcal{B}_{F_{n}}(X)|}{|F_{n}|},
	$$
	where $\mathcal{B}_{F_{n}}=\{y_{|F_{n}}\colon y\in X\}$ are words in $X$ over $F_{n}$.
\end{lem}
\begin{proof}
	Taking into account Lemma \ref{characterisation} it is enough to show inequality for sofic approximation of the form $\Sigma=\{ \sigma_{n}\colon G\to\Sym(\amalg _{f_{n}}G/H_{n}) \}$ for some $(f_{n})_{n\in\nat}\subset\nat$. Let $\epsilon>0$. Consider $F_{N}\subset G$, for some $N\in\nat$, and its associated open set  $\mathcal{U}_{F_{N}}=\bigcup_{w \in \mathcal{B}_{F_{N}}(X)}[w]$, $\delta >0$, and $n\in\mathbb{N}$ big enough so that $N<n$. Take any open neighborhood $\Uee$ of $X$ such that $\Uee\subset \Uee_{F_{N}}$. 
	
	Let $l\leq n$. We adopt the following notation: for $c\in \tilde{F}_{l}:=\amalg_{f_{n}}F_{l}$ we will write $cH_{l}:=(\iota(c),\kappa(c)H_{l})$. For any $c\in \tilde{F}_{l}$ the set $cH_{l}$ will be called a coset of $H_{l}$. Note that $G$ acts naturally on $G/H_{l}$ through left and right multiplication. As $G_{l}:=\amalg _{f_{n}}G/H_{l}$ can be considered as $f_{n}$ copies of $G/H_{l}$, it is natural to extend the action of $G$ on $G/H_{l}$ to the disjoint union. Let us write $g(cH_{l})=gcH_{l}$ and $(cH_{l})g=cgH_{l}$ for $c\in \tilde{F}_{l}$ and $g\in G$.
	We will write $(g\phi)(v)=\phi(vg^{-1})$ for $\phi\in k^{G_{l}}$, $g\in G$ and $v\in G_{l}$.
	Now suppose we are given a $(\sigma_{n},\delta,\mathcal{U}_{F_{N}})$-microstate $\psi\colon G_{n}\to k$. Let us compute $\Pi_{v}^{\sigma_{n}}(\psi)\in k^{G}$ for $v=cH_{n}\in G_{n}$, where $c\in \tilde{F}_{n}$:
	$$
	\Pi_{v}^{\sigma_{n}}(\psi)(g):=\psi(\sigma_{n}(g)^{-1}(v))=\psi(cgH_{n})=g^{-1}\psi(cH_{n}).
	$$
	Define $ S_{\psi}:=\{v\in G_{n}\colon \Pi_{v}^{\sigma_{n}}(\psi)\in \mathcal{U}_{F_{N}}\} $. Since $\psi$ is a $(\sigma_{n},\delta,\mathcal{U}_{F_{N}})$-microstate, by definition
	\begin{align*}
	|S_{\psi}|&=|\{v\in G_{n}\colon \Pi_{v}^{\sigma_{n}}(\psi)\in \mathcal{U}_{F_{N}}\}|=
	|\{cH_{n}\in G_{n}\colon c\in \tilde{F}_{n} \textsl{ and } \Pi_{cH_{n}}^{\sigma_{n}}(\psi)\in \mathcal{U}_{F_{N}}\}|\\
	&=|\{cH_{n}\in G_{n}\colon c\in \tilde{F}_{n} \textsl{ and } (g^{-1}(\psi(cH_{n})))_{g\in G}\in \mathcal{U}_{F_{N}}\}|\\
	&=|\{cH_{n}\in G_{n}\colon c\in \tilde{F}_{n} \textsl{ and } (g^{-1}(\psi(cH_{n})))_{g\in F_{N}}=w\textsl{ for some }w\in\mathcal{B}_{F_{N}}(X)\}|
	\\&\geqslant (1-\delta)|\tilde{F}_{n}|.
	\end{align*}
	Denote by $E_{\psi,i}=\iota^{-1}(i)\cap S_{\psi}$ the set of elements from $S_{\psi}$ that are in the $i$-th copy of $G/H_{n}$ for $i\leq f_{n}$. Note that on average there are $|E_{\psi,i}|/|F_{N}|$ elements from $E_{\psi,i}$ in a coset of $H_{N}$. Hence for every $i\leq f_{n}$ there exists an element $c_{\psi,i}\in \tilde{F}_{N}$ such that\[ 
	R_{\psi,i}:=\{c\in E_{\psi,i}~|~cH_{n}\subset c_{\psi,i}H_{N}\text{ and }(g^{-1}(\psi(cH_{n})))_{g\in F_{N}}=w\textsl{ for some }w\in\mathcal{B}_{F_{N}}(X)\}
	\] with $ |R_{\psi,i}|\geq \frac{|E_{\psi,i}|}{|F_{N}|}. $ Note that every $(\sigma_{n},\delta,\mathcal{U}_{F_{N}})$-microstate is determined on the set of cardinality at least $(1-\delta)|F_{n}|f_{n}$ by choosing for every $i\leq f_{n}$ and every $c\in R_{\psi,i}$ a word $w_{\psi,c}\in \mathcal{B}_{F_{N}}(X)$ and putting $(g^{-1}(\psi(cH_{n})))_{g\in F_{N}}=w_{\psi,c}$.
	
	Let $\mathcal{S}\subset G_{n}$ with $|\mathcal{S}|=(1-\delta)|F_{n}|f_{n}$. By the discussion above, for every $i\leq f_{n}$ there exists a set $R_{i}\subset \mathcal{S}_{i}:=\iota^{-1}(i)\cap \mathcal{S}\subset\mathcal{S}$ with cardinality $\lceil|\mathcal{S}_{i}|/|F_{N}|\rceil$ which is contained in some coset of $H_{N}$, that is $R_{i}\subset c(H_{N}/H_{n})$ for some $c\in\tilde{F}_{N}\cap\iota^{-1}(i)$. This choice defines a function 
	\[ \Theta\colon\{\mathcal{S}\subset G_{n} \text{ with } |\mathcal{S}|\geq(1-\delta)|F_{n}|f_{n}\}\ni\mathcal{S}\mapsto\Theta(\mathcal{S})=(R_{i})_{i=1}^{f_{n}}\subset\prod_{i=1}^{f_{n}}\iota^{-1}(i). \] Put $\mathcal{S}^{\Theta}=G_{n}\setminus \bigcup_{i=1}^{f_{n}}(R_{i}(\tilde{F}_{N}\cap\iota^{-1}(i)))$. It is a matter of a simple check that 
\[ 	|\bigcup_{i=1}^{f_{n}}R_{i}(\tilde{F}_{N}\cap\iota^{-1}(i))|=\sum_{i=1}^{f_{n}}|R_{i}||\tilde{F}_{N}\cap\iota^{-1}(i)|=\sum_{i=1}^{f_{n}}\lceil|\mathcal{S}_{i}|/|F_{N}|\rceil|\tilde{F}_{N}\cap\iota^{-1}(i)|\geq(1-\delta)|F_{n}|f_{n}, \]
	so $$
	|\mathcal{S}^{\Theta}|=|F_{n}|f_{n}-\sum_{i=1}^{f_{n}}|R_{i}||\tilde{F}_{N}\cap\iota^{-1}(i)|\leq \delta|F_{n}|f_{n}.
	$$ With the function $\Theta$ defined above, let us introduce the map
\begin{multline*}
	\Psi\colon\coprod_{\substack{\mathcal{S}\subset G_{n},\\|\mathcal{S}|=\lceil(1-\delta)|F_{n}|f_{n}\rceil}}\big(\prod_{i=1}^{f_{n}}\mathcal{B}_{F_{N}}(X)^{R_{i}}\big)\times k^{\mathcal{S}^{\Theta}}\ni ((w_{c})_{c\in R_{i},i=1,...,f_{n}},\gamma)\mapsto\\ \Psi((w_{c})_{c\in R_{i},i=1,...,f_{n}},\gamma)=\psi\in \Omega(\sigma_{n}, \delta, \mathcal{U}_{F_{N}}),
\end{multline*}	where $(\psi(cgH_{n}))_{g\in F_{N}}=w_{c}$ for $c\in R_{i}$, $i=1,...,f_{n}$ and $\psi(v)=\gamma(v)$ for $v\in \mathcal{S}^{\Theta}$.
	We will show that $\Psi$ is surjective. Fix $\psi \in \Omega(\sigma_{n}, \delta, \mathcal{U}_{F_{N}})$. As before, let $ S_{\psi}:=\{v\in G_{n}\colon \Pi_{v}^{\sigma_{n}}(\psi)\in \mathcal{U}_{F_{N}}\} $. Function $\Theta$ gives us sets $\Theta(S_{\psi})$ and $(S_{\psi})^{\Theta}$. Let $w_{\psi,c}:=(g^{-1}(\psi(cH_{n})))_{g\in F_{N}}$ for $c\in \Theta(S_{\psi})_{i}$ and $1\leq i\leq f_{n}$, moreover put $\gamma_{\psi}(v):=\psi(v)$ for $v\in (S_{\psi})^{\Theta}$. Our previous discussion implies that $\Psi((w_{\psi,c})_{c\in \Theta(S_{\psi})_{i},i=1,...,f_{n}},\gamma_{\psi})=\psi$. Hence $\Psi$ is surjective and we can estimate
\begin{multline*}
	|\Omega(\sigma_{n}, \delta, \mathcal{U}_{F_{N}})|\leq
	|\coprod_{\substack{\mathcal{S}\subset G_{n},\\|\mathcal{S}|=\lceil(1-\delta)|F_{n}|f_{n}\rceil}}\big(\prod_{i=1}^{f_{n}}\mathcal{B}_{F_{N}}(X)^{R_{i}}\big)\times k^{\mathcal{S}^{\Theta}}|\leq\\
	{{|F_{n}|f_{n}}\choose{\lceil(1-\delta)|F_{n}|f_{n}\rceil}} \big(\prod_{i=1}^{f_{n}}|\mathcal{B}_{F_{N}}(X)|^{|R_{i}|}\big) k^{|\mathcal{S}^{\Theta}|}|.
\end{multline*}Let us use Lemma \ref{lemat-n!} to bound ${{|F_{n}|f_{n}}\choose{\lceil(1-\delta)|F_{n}|f_{n}\rceil}}$ by $e^{\beta(\delta)|F_{n}|f_{n}}$, where $\beta(\delta)\to 0$ as $\delta\to0$. Additionally, taking the logarithm of both sides and  dividing by $|F_{n}|f_{n}$, we obtain
\begin{align*}
\frac{1}{|F_{n}|f_{n}} \log |\Omega(\sigma_{n}, \delta, \mathcal{U}_{F_{N}})|&\leq
\beta(\delta)+\frac{1}{|F_{n}|f_{n}} \log|\mathcal{B}_{F_{N}}(X)|\sum_{i=1}^{f_{n}}|R_{i}|+ \log k\frac{ |\mathcal{S}^{\Theta}|}{|F_{n}|f_{n}}|\\
 &\leq\beta(\delta)+\frac{1}{|F_{n}|f_{n}} \log|\mathcal{B}_{F_{N}}(X)|\sum_{i=1}^{f_{n}}
 |\mathcal{S}_{i}|/|F_{N}|+ \delta\log k\\
 &\leq\beta(\delta)+\frac{1}{|F_{n}|f_{n}} \frac{\log|\mathcal{B}_{F_{N}}(X)|}{|F_{N}|}
 |\mathcal{S}|+ \delta\log k\\
  &\leq\beta(\delta)+ (1-\delta)\frac{\log|\mathcal{B}_{F_{N}}(X)|}{|F_{N}|}
 + \delta\log k.
\end{align*}As a result we get
\[
\frac{1}{|F_{n}|f_{n}} \log |\Omega(\sigma_{n}, \delta, \mathcal{U}_{F_{N}})|\leq
\beta(\delta)+ (1-\delta)\frac{\log|\mathcal{B}_{F_{N}}(X)|}{|F_{N}|}
+ \delta\log k.
 \]
Now, let $N\to\infty$, $n\to\infty$ and apply infimum over all open neighborhoods  $\mathcal{U}$ of $X$,  and finally $\delta\to0$ to get
\[
\inf_{\delta>0}\inf_{\mathcal{U}\supset X}\limsup_{n\to\infty}\frac{1}{|F_{n}|f_{n}} \log |\Omega(\sigma_{n}, \delta, \mathcal{U}_{F_{N}})|\leq
\liminf_{N\to\infty}\frac{\log|\mathcal{B}_{F_{N}}(X)|}{|F_{N}|}.\qedhere
\]
\end{proof}
\section{Toeplitz elements}\label{toeplitz-elements}
	To start working with Toeplitz subshift, we need to introduce some additional tools. All the definitions from this section can be found in \cite{Krieger}. Let $k\in\nat$.
	\begin{dfn}
		Let $H$ be a normal subgroup of $G$ and $x\in k^{G}$. \emph{The $H$-Periodic part} of $x$ is defined as
		$$
		\per_{H}(x)=\{g \in G:  x_{h^{-1}g}=x_{g}\textrm{ for }h\in H\}.
		$$
		Moreover, for $i\in k$, we introduce sets
		$$
		\per_{H}(x,i)=\{g \in \per_{H}(x):  x_{g}=i\}.
		$$ Note that $\per_{H}(x,i)$ is a sum of all cosets of $H$ on which $x$ is constant. We will denote the cardinality of the set of such cosets by $\#\per_{H}(x,i)$.
	\end{dfn}
\begin{dfn}\label{6.2}
	Let $\{H_{n}\}_{n\in\nat}$ be a decreasing sequence of finitely indexed subgroups of $G$. We call an element $x\in k^{G}$ a Toeplitz element with respect to $\{H_{n}\}_{n\in\nat}$ if $G=\bigcup_{n\in\nat}\per_{H_{n}}(x)$.
\end{dfn}
\begin{dfn}	
	An element $S_{H}(x) \in \tilde{k}^{G}$, where $\tilde{k}:=k\cup \{*\}$, is called \emph{$H$-skeleton} of $x$ if
	\[
	{(S_{H}(x))}_{g}=\left\{ \begin{array}{ll} 
	x_g & \textrm{ for } g \in \per _{H},\\
	* & \textrm{ otherwise}.
	\end{array} \right.
	\]
	If for some $g\in G$ we have ${(S_{H}(x))}_{g}=*$, then $g$ is said to be in a hole of $S_{H}(x)$. We extend previous Definition \ref{6.2} by setting $\#\per_{H}(x,*):=\#\per_{H}(S_{H}(x),*)$.The $H$-holes are the cosets on which $S_{H}(x)\equiv*$. If $S_{H}(x)|_{v}=j$ for some $v\in G/H$ and $j\in k$, then we will say that $H$ is coloured in $j$.
\end{dfn}

We now state an interesting characterisation of residually finite groups, which is due to Fabrice Krieger. The following theorem is the reason why we are interested mainly in Toeplitz elements over residually finite groups, and not just sofic groups.
\begin{thm}\cite[Cor. 4.3.]{Krieger}
	Let $G$ be an infinite countable group and $k\geq2$. Then $G$ is residually finite if and only if there exists a Toeplitz element in $k^{G}$ with trivial stabilizer.
\end{thm}
\section{Entropy of Toeplitz systems}\label{Entropy of Toeplitz systems}
We are ready to formulate and prove main theorem.
\begin{lem}\label{lem1}
	Let $x \in k^G$ be a Toeplitz element and $\{H_n\}_{n=1}^{\infty}$, $\{L_n\}_{n=1}^{\infty}$ be non increasing sequences of finitely indexed subgroups of $G$, where $L_n \subset H_n$, then the limits $\lim_{n\to\infty}\frac{\#\per _{H_n}(x,*)}{[G:H_n]}$ and $\lim_{n\to\infty}\frac{\#\per _{L_n}(x,*)}{[G:L_n]}$ exists and satisfy $\lim_{n\to\infty}\frac{\#\per _{L_n}(x,*)}{[G:L_n]} \leq \lim_{n\to\infty}\frac{\#\per _{H_n}(x,*)}{[G:H_n]}$.
\end{lem}
\begin{proof}
	Notice that the sequence $\{\frac{\#\per _{H_n}(x,*)}{[G:H_n]}\}_{n=1}^{\infty}$ is non increasing, since for each $n\in \mathbb{N}$ we have 
	$$\frac{\#\per _{H_n}(x,*)}{[G:H_n]}\leq \frac{\#\per _{H_{n-1}}(x,*)[H_{n-1}:H_{n}]}{[G:H_{n-1}][H_{n-1}:H_{n}]}\leq\frac{\#\per _{H_{n-1}}(x,*)}{[G:H_{n-1}]},$$ 
	so both limits exists. Moreover 
	$$\lim_{n\to\infty}\frac{\#\per _{L_n}(x,*)}{[G:L_n]} \leq \lim_{n\to\infty}\frac{\#\per _{H_n}(x,*)[H_{n}:L_{n}]}{[G:H_n][H_{n}:L_{n}]}=\lim_{n\to\infty}\frac{\#\per _{H_n}(x,*)}{[G:H_n]}.\qedhere$$
\end{proof}
\begin{dfn}
	Let $G=\langle S\rangle$ be a free group generated by a finite set $S$ and let $\Sigma=\{\sigma\colon G\to \Sym(V_{n})\}_{n\in\nat}$ be a sofic approximations. For every $n\in \nat$ we define the stabilizer of $\sigma_{n}$ by $\stab
	\sigma_{n}:=\bigcap_{v\in V_{n}}\stab_{\sigma_{n}} (v)$, where
	\begin{multline*}
	\stab_{\sigma_{n}} (v):=\{g\in\langle S\rangle~|~\sigma_{n}(s_{1})^{\alpha_{1}}\circ\ldots\circ\sigma_{n}(s_{m})^{\alpha_{m}}v=v\text,\\\text{ where } g=\prod_{i=1}^{m}s_{i}^{\alpha_{i}}\in G\text{ is a reduced word in }\langle S\rangle,\text{for some }m\in\nat\}.
	\end{multline*} 
\end{dfn}
We are ready to formulate the main result. It is a sofic counterpart of \cite[Lem. 5.4.]{KriegerTheorem}.
\begin{thm}\label{main}
	Suppose that $G$ is a residually finite group. Let $x \in k^G$ be a Toeplitz element, with respect to the sequence of finitely indexed normal subgroups $H_n \searrow \{e\}$, put $X=\overline{Gx}$. Then we have
\begin{equation}\label{bound-entropy}
	h_{\Sigma}(X,G) \leq \liminf_{n\to\infty}\frac{1}{[G:H_n]}\sum_{v\in G/H_{n}}\log |x(v)|
	\leq \lim_{n\to\infty}\frac{\#\per _{H_n}(x,*)}{[G:H_n]}\log k,
\end{equation}
	for any sofic approximation $\Sigma=\{\sigma_{n}\}_{n\in\nat}$ by homomorphisms such that $\ker \sigma_{n}\subset H_{n}$ for every $n\in\nat$. Additionally, in the case of a free group $\langle S\rangle$ generated by $S$, the inequality holds for any sofic approximation $\Sigma=\{\sigma_{n}\}_{n\in\nat}$ such that $\stab \sigma_{n}\subset H_{n}$ for every $n\in\nat$.
\end{thm}
\begin{proof}
	Since $\ker \sigma_{n}\subset H_{n}$, by Lemma \ref{lem1} it suffices to prove the theorem for a Toeplitz element with respect to $\{\ker\sigma_{n}\}_{n\in\nat}$. By Lemma \ref{characterisation} we know that there exists a sequence $(f_{n})_{n\in\nat}\subset \nat$ such that $\Gamma=\{\gamma\colon G\to\Sym(\amalg_{f_{n}}G/\ker\sigma_{n})\}_{n\in\nat}$ is a sofic approximation sequence and $h_{\Sigma}(X,G)\leq h_{\Gamma}(X,G)$. Choose a telescoping increasing sequence of fundamental domains $\mathcal{F}=\{F_{n}\}_{n\in\mathbb{N}}$ corresponding to $\{G/\ker\sigma_{n}\}_{n\in\nat}$ such that $\bigcup \mathcal{F}=G$. It follows from Lemma \ref{sofic-residually} that $h_{\Gamma}(X,G)\leq\liminf_{n\to\infty}\frac{\log |\mathcal{B}_{F_{n}}(X)|}{|F_{n}|}$. It is clear, that $|\mathcal{B}_{F_{n}}(X)|\leqslant |F_{n}|\prod_{v\in G/H_{n}}|x(v)|$. Recall that $v\in G/H_{n}$ is a set and by $x(v)$ we understand the image of $v$ through $x\in k^{G}$. Taking exponential growth we obtain desired inequality.
	
	To cope with the case of a free group, note that by Theorem \ref{free group characterisation} and Lemma \ref{free-group-sofic-app} for any sofic approximation $\Sigma$ we can construct a sofic approximation by homomorphisms $\ddot{\Sigma}=\{\ddot{\sigma}_{n}\}_{n\in\nat}$ such that $h_{\Sigma}(X,G)=h_{\ddot{\Sigma}}(X,G)$ and $\ddot{\sigma}_{n}(s)=\sigma_{n}(s)$ for every $s\in S$ and $n\in\nat$. Consequently, if we write every $g\in G$ as a reduced word, we can prove that
	\[ \stab_{\ddot{\sigma}_{n}}(v)=\{g\in\langle S\rangle~|~\ddot{\sigma}_{n}(g)v=v\}=\stab_{\sigma_{n}}(v) \]
	and we can express the kernel of $\ddot{\sigma}_{n}$ as
	\[ 
	\ker\ddot{\sigma}_{n}=\bigcap_{v\in V_{n}}\stab_{\ddot{\sigma}_{n}}(v)=\bigcap_{v\in V_{n}}\stab_{\sigma_{n}}(v)= \stab \sigma_{n}.
	 \]Now, the case of free groups follows from our earlier discussion.
\end{proof}
Inequality \eqref{bound-entropy} is the best possible in the sense of the following theorem.
\begin{thm}\label{toeplitz-example}
	If $H_n\searrow \{e\}$ is a sequence of finitely indexed normal subgroups of $G$, 
	$\{a_n\}_{n=1}^{\infty}$ is a sequence of positive numbers satisfying $a_n<a_{n-1}[H_{n-1}:H_{n}]$, the limit $\lim_{n\to\infty}\frac{a_n}{[G:H_n]}=\theta$ exists and satisfies $\theta<1$, then for any $k>1$ there exists a Toeplitz element $x\in k^{G}$, such that $$h_{\Sigma}(\overline{Gx},G)= \lim_{n\to\infty}\frac{a_n}{[G:H_n]}\log k=\lim_{n\to\infty}\frac{\#\per _{H_n}(x,*)}{[G:H_n]}\log k,$$where $\Sigma=\{\sigma_{n}\}_{n\in\nat}$ is a sofic approximation given by $ \sigma_{n}\colon G\to \Sym(G/H_{n})$ for every $n\in\nat$ and $\#\per _{H_n}(x,*)=a_{n}$ for every $n\in\nat$.
\end{thm}
\begin{proof}
	Note that taking subsequence of $\{H_{n}\}_{n\in\nat}$, $\{a_{n}\}_{n\in\nat}$ and $\Sigma$, we do not change neither the entropy $h_{\Sigma}(Y,G)$, for any system $(Y,G)$, nor the limit $\lim_{n\to\infty}\frac{a_n}{[G:H_n]}$. If $\theta=0$, we take as a Toeplitz element constant sequence. Assume that $\theta>0$. Therefore, $a_{n}$ cannot be bounded, so taking a subsequence of $\{H_{n}\}_{n\in\nat}$, $\{a_{n}\}_{n\in\nat}$ and $\Sigma$ if necessary, we can assume that $a_{n}>\omega_{n}$ for $\omega_{n}=2k^{n}|F_{n}|$ and $[H_{n}:H_{n+1}]>k$ for every $n\in\nat$.	Number $a_n$ is supposed to indicate, at least asymptotically, cardinality of holes in the $n$-th step of the construction of $x$. Condition $a_n<a_{n-1}[H_{n-1}:H_{n}]$ is equivalent to the statement that, if on the $(n-1)$-th step there was $a_{n-1}$ holes of $H_{n-1}$, then during $n$-th step we ought to colour at least one hole of $H_{n}$. Hence, choosing appropriate subsequences, we can assume that in the $n$-th step we can colour $k^{n}|F_{n}|$ holes, analytically it means that $k^{n}|F_{n}|\leq a_{n-1}[H_{n-1}:H_{n}]-a_n$.
	
	Let $\{F_{n}\}_{n\in\nat}$ be a sequence of fundamental domains of groups $\{H_{n}\}_{n\in\nat}$, such that for every $n\in\nat$ there is $F_{n}=F_{n-1,n}F_{n-1}$, where $F_{n-1,n}$ is a fundamental domain of $H_{n-1}/H_{n}$.
	
	Enumerate G as $G=\{g_l\}_{l=1}^{\infty}$. It is enough to define $S_{H_{n}}(x)$ for each $n\in\mathbb{N}$. We proceed inductively. Let $\beta_{1}=a_{1}$ and $\beta_{l+1}=\beta_{l}+ a_{\beta_{l}}+m_{l}$, where $(m_{l})_{l\in\nat}$ will be specified later. We will inductively define a Toeplitz element $x$, defining for each $l\in\mathbb{N}$ $H_{i}$-skeleton of $x$ for $i \in (\beta_{l},\beta_{l+1}]$. We will also define the family $W^{l}=\{w_{\gamma}^{l}\}_{\gamma\in k ^{\mathcal{W}_{l}}}\subset F_{\beta_{l}+m_{l},\beta_{l+1}}$, where $\mathcal{W}_{l}=\{s_j\}_{j=1}^{a_{\beta_{l}}}\subset F_{\beta_{l}}$ is the set containing all those $s\in F_{\beta_{l}}$ such that $S_{H_{\beta_{l}}}(x)_{|sH_{\beta_{l}+m_{l}}}=*$. The family $W^{l}$ will have the property, that for every $\gamma\in k ^{\mathcal{W}_{l}}$ and $s\in \mathcal{W}_{l}$ we have $S_{H_{\beta_{l+1}}}(x)|_{w_{\gamma}^{l}s H_{\beta_{l+1}}}=\gamma(s)$.
	Assume, that $S_{H_{\beta_{l}}}(x)$ is defined. Let $Q_{l}=F_{\beta_{l}}^{2}\setminus F_{\beta_{l}}$ be the set of all elements $q\in F_{\beta_{l}}^{2}\setminus F_{\beta_{l}}$ such that $S_{H_{\beta_{l}}}(x)|_{q H_{\beta_{l}}}= *$. Choose $m_{l}\in\nat$ big enough so that $Q_{l}\subset F_{\beta_{l}+m_{l}}$, in this way, if indeed there will be $W^{l}\subset F_{\beta_{l}+m_{l},\beta_{l+1}}$, then for every distinct pairs $(w_{\gamma}^{l},\eta),(w_{\gamma'}^{l},\eta')\in W^{l}\times Q_{l}$ we will have $w_{\gamma}^{l}H_{\beta_{l+1}}\eta\cap w_{\gamma}^{l}H_{\beta_{l+1}}\eta=\emptyset$ and $S_{H_{\beta_{l}}}(x)|_{w_{\gamma}^{l} q H_{\beta_{l}}}= *$, since $w_{\gamma}^{l}\in H_{\beta_{l}+m_{l}}$
	
	Choose $n_{l}\in \mathbb{N}$ to be the least number such that $g_{l}$ lies in a hole of $S_{H_{\beta_{l}}}(x)$, rearranging numeration we can assume that $g_{l}\in  s_1  H_{\beta_{l}}$. Put $\tilde{\beta}_{l}:=\beta_{l}+m_{l}$.

	We define $S_{H_{\tilde{\beta}_{l}+1}}(x)$ by choosing $k$ cosets $ q_{\gamma_{1}}^{1}H_{\tilde{\beta}_{l}+1}s_1\subset   H_{\tilde{\beta}_{l}}s_1,$ for ${\gamma_{1}}\in[1,k]$ and placing there respectively ${\gamma_{1}}$, where $q^{1}_{\gamma_{1}}\in F_{\tilde{\beta}_{l},\tilde{\beta}_{l}+1}$, then we colour $g_{n_{l}}H_{\tilde{\beta}_{l}+1}$ whatever we like. Moreover, the cosets  
	\[ q_{\gamma_{1}}^{1}H_{\tilde{\beta}_{l}+1}s_2,  q_{\gamma_{1}}^{1}H_{\tilde{\beta}_{l}+1}s_3 ,\ldots, q_{\gamma_{1}}^{1}H_{\tilde{\beta}_{l}+1}s_{a_{\tilde{\beta}_{l}}},q_{\gamma_{1}}^{1}H_{\tilde{\beta}_{l}+1}\eta, \]
	for ${\gamma_{1}}\in[1,k]$ and $\eta\in Q_{l}$, must remain holes in this step. Note that $sH_{\tilde{\beta}_{l}}\cap \eta H_{\tilde{\beta}_{l}}=\emptyset$ for every $\eta\in Q_{l}$ and $s\in \Wee_{l}$, since $Q_{l}\cup\Wee_{l}\subset F_{\beta_{l}+m_{l}}$ and $Q_{l}\cap \Wee_{l}\subset Q_{l}\cap F_{\beta_{l}}=\emptyset$.  In this way, we demand that there are $k |Q_{l}|+k|\Wee_{l}|+1\leq 2k|F_{\tilde{\beta}_{l}}|<\omega_{\tilde{\beta}_{l}+1}$ $H_{\tilde{\beta}_{l}+1}$-holes, which is possible, since $\omega_{\tilde{\beta}_{l}+1}<a_{\tilde{\beta}_{l}+1}$. Notice also, that we have coloured $k$ $H_{\tilde{\beta}_{l}+1}$-cosets. Therefore, we choose some other $H_{\tilde{\beta}_{l}+1}$-cosets and place there $1$, so that in this turn we have $a_{\tilde{\beta}_{l}+1}$ $H_{\tilde{\beta}_{l}+1}$-holes. 
	
	We proceed with the second step. We will define $S_{H_{\tilde{\beta}_{l}+2}}(x)$. We choose some $k^2$ cosets \[ q_{\gamma_{2}}^{2}q_{\gamma_{1}}^{1}H_{\tilde{\beta}_{l}+2}s_2 \subset  q_{\gamma_{1}}^{1}H_{\tilde{\beta}_{l}+1}s_2\subset   H_{\tilde{\beta}_{l}}s_2 \] from	 $\per_{H_{\tilde{\beta}_{l}+2}}(S_{H_{\tilde{\beta}_{l}+2}}(x),*)$ for some $q_{\gamma_{2}}^{2}\in F_{\tilde{\beta}_{l}+1,\tilde{\beta}_{l}+2} $, where $\gamma_{1} ,\gamma_{2} \in [1,k]$, and place there $\gamma_{2}$. Moreover, the cosets \[ q_{\gamma_{2}}^{2}q_{\gamma_{1}}^{1} H_{\tilde{\beta}_{l}+2}s_t,q_{\gamma_{2}}^{2}q_{\gamma_{1}}^{1}H_{\tilde{\beta}_{l}+2}\eta, \] for $\gamma_{2},\gamma_{1}\in[1,k], t\in[3,a_{\beta_{l}}]$ and for every $\eta\in Q_{l}$, must remain holes. Again, as before we have demanded, that $k^{2}(|\Wee_{l}|-1)+k^{2}|Q_{l}|\leq 2k^{2}|F_{\tilde{\beta}_{l}}|<\omega_{\tilde{\beta}_{l}+2}$ cosets must be $H_{\tilde{\beta}_{l}+2}$-holes, it is possible by the inequality $\omega_{\tilde{\beta}_{l}+2}<a_{\tilde{\beta}_{l}+2}$. In this step, we have coloured $k^2$ cosets, hence we choose some cosets and place there $2$, so that we have $a_{\tilde{\beta}_{l}+2}$ holes.	
	
	In the $j$-th step, for $j\leq a_{\beta_{l}}$, we choose from $\per_{H_{\tilde{\beta}_{l}+j}}(S_{H_{\tilde{\beta}_{l}+j}}(x),*)$ some $k^j$ cosets \[ q_{\gamma_{j}}^{j}q_{\gamma_{j-1}}^{j-1}\ldots q_{\gamma_{1}}^{1} H_{\tilde{\beta}_{l}+j}s_{j}\subset  q_{\gamma_{j-1}}^{j-1}\ldots q_{\gamma_{1}}^{1}   H_{\tilde{\beta}_{l}+j-1}s_{j}\subset  H_{\tilde{\beta}_{l}}s_{j} \] with $q_{\gamma_{f}}^{j} \in F_{\tilde{\beta}_{l}+j-1,\tilde{\beta}_{l}+j}$ and we place there $\gamma_{j}$ for every $\gamma_{j}\in[1,k]$. Moreover, the cosets 
	\[ 	q_{\gamma_{j}}^{j} q_{\gamma_{j-1}}^{j-1} \ldots q_{\gamma_{1}}^{1}  H_{\tilde{\beta}_{l}+j}s_{t},q_{\gamma_{j}}^{j} q_{\gamma_{j-1}}^{j-1} \ldots q_{\gamma_{1}}^{1}H_{\tilde{\beta}_{l}+j}\eta, \] for every $t\in[j+1,a_{\beta_{l}}]$, $\gamma_{b}\in[1,k]$, for $b=1,...,j$, and for every $\eta\in Q_{l}$, must be omitted. Note that, we have coloured $k^j$ cosets, which is possible and we demanded that \[ k^{j}(a_{\beta_{l}}-j-1)+k^{j}|Q_{l}|\leq 2k^{\tilde{\beta}_{l}+j}|F_{\tilde{\beta}_{l}+j}|\leq\omega_{\tilde{\beta}_{l}+j} <a_{\tilde{\beta}_{l}+j} \] cosets must remain $H_{\tilde{\beta}_{l}+j}$-holes. We choose, additionally, some cosets and place there $j \mod k$, so that in this turn we have  $a_{\tilde{\beta}_{l}+j}$ $H_{\tilde{\beta}_{l}+j}$-holes. Denote $w_{\gamma}^{l}=   q_{\gamma_{a_{\beta_{l}}}}^{a_{\beta_{l}}}q_{\gamma_{a_{\beta_{l}}}}^{a_{\beta_{l}}-1}\ldots q_{\gamma_{1}}^{1}  \in F_{\beta_{l}+m_{l},\beta_{l+1}}$, for $\gamma \in k^{a_{\beta_{l}}}$.
	
	We now consider $k^{\beta_{l}}$ as $k^{\mathcal{W}_{l}}$, where $\mathcal{W}_{l}=\{s_j\}_{j=1}^{a_{\beta_{l}}}\subset F_{a_{\beta_{l}}}$, via identification $j\mapsto s_{j}\in F_{a_{\beta_{l}}}$. It follows directly from construction that $((w_{\gamma}^{l})^{-1}x)|_{\mathcal{W}_{l}}=\gamma\in k^{\mathcal{W}_{l}}$.
	
	In this way, we can carry on with induction to the $a_{\beta_{l}}$-step. Note that in every step we have omitted cosets $w_{\gamma}^{l}H_{\beta_{l+1}}\eta$, for every $\gamma \in k^{a_{\beta_{l}}}$ and $\eta\in Q_{l}$. Since, as we have noted before, for every distinct pairs $(w_{\gamma}^{l},\eta),(w_{\gamma'}^{l},\eta')\in W^{l}\times Q_{l}$ holds $w_{\gamma}^{l}H_{\beta_{l+1}}\eta\cap w_{\gamma'}^{l}H_{\beta_{l+1}}\eta'=\emptyset$, we can colour aforementioned cosets independently. Therefore, put $S_{H_{\beta_{l+1}}}(x)|_{w_{\gamma}^{l}sH_{\beta_{l+1}}}=\gamma(s)$ on $w_{\gamma}^{l}\eta H_{\beta_{l+1}}$ if $sH_{\beta_{l}}=\eta H_{\beta_{l}}$ for $\eta\in Q_{l}$ and $s\in F_{\beta_{l}}$. It is not hard to see, that in the last step, where we define $S_{H_{\beta_{l+1}}}(x)$ we will colour $k^{a_{\beta_{l+1}}}$ $H_{\beta_{l+1}}$-cosets. Therefore adding to that cosets $w_{\gamma}^{l}H_{\beta_{l+1}}\eta$, for $\gamma \in k^{a_{\beta_{l}}}$ and $\eta\in Q_{l}$ we will colour at most $2k^{a_{\beta_{l+1}}}|F_{\beta_{l}}|\leq k^{a_{\beta_{l+1}}}|F_{\beta_{l+1}}|$ cosets, which is possible. Hence, also in this last step we can demand that we have exactly $a_{\beta_{l+1}}$  $H_{\beta_{l+1}}$-holes.
	
	Procedure of choosing the least element from $G$ which has not been chosen yet reassures us that the constructed element $x\in k ^{G}$ is a Toeplitz element. We are ready to compute sofic entropy of $(\overline{Gx},G)$. We set $X=\overline{Gx}$. For every $n\in\nat$, substitute $K_{n}=H_{\beta_{n}}$, analogously consider $L_{n}= F_{\beta_{n}}$. We have to consider a new sofic approximation $\tilde{\Sigma}=\{\tilde{\sigma}_{n}:=\sigma_{\beta_{n}}\colon G\to\Sym(G/K_{n})\}_{n\in\nat}$, which is equivalent to $\Sigma$. 
	
	We will compute sofic entropy with respect to $\tilde{\Sigma}$. Note that by Theorem \ref{main}, since $\#\per _{K_n}(x,*)=a_{\beta_{n}}$ and $[G:K_{n}]=[G:H_{\beta_{n}}]$, we have $h_{\tilde{\Sigma}}(\overline{Gx},G)\leq \lim_{n\to\infty}\frac{a_n}{[G:H_n]}\log k$. It remains to prove the reverse inequality. To do so we will use more general Definition \ref{entropy-def1}, as it is more suitable in this situation.
	
	Let $\epsilon>0$. Fix a finite set $F\subset G$, $\delta>0$ and $n\in\nat$ big enough so that $F\subset L_{n}$. For $\gamma\in k^{\Wee_{n}}$ let us define $\phi_{\gamma}\colon G/K_{n}\to X$ by the formula $\phi_{\gamma}(i,g K_{n})=g^{-1}(w_{\gamma}^{n})^{-1}x$, where $g\in L_{n}$. We will show that $\phi_{\gamma}$ is an $(F\cup F^{-1},\delta)$-pseudoorbit. Fix $\gamma\in k^{\Wee_{n}}$, and $g\in L_{n}$, $s\in F$. Assume that $gs\in L_{n}=F_{\beta_{n}}$. Let us compute
	\begin{align*}
	\phi_{\gamma}\tilde{\sigma}_{s^{-1}}(gK_{n})=\phi_{\gamma}(gsK_{n})=s^{-1}g^{-1}(w_{\gamma}^{n})^{-1}x=s^{-1}\phi_{\gamma}(i,gK_{n}),
	\end{align*}hence, in this case, $d(\phi_{\gamma}\tilde{\sigma}_{s^{-1}}(gK_{n}),s^{-1}\phi_{\gamma}(gK_{n}))=0$. On the other hand, let $gs\notin L_{n}$, then $gs\in Q_{n}$ and 
	\[ (s^{-1}\phi_{\gamma}(i,gK_{n}))|_{K_{n+1}}=(s^{-1}g^{-1}(w_{\gamma}^{n})^{-1}x)|_{K_{n+1}}=x|_{w_{\gamma}^{n}gsK_{n+1}}=x|_{w_{\gamma}^{n}pK_{n+1}}, \]
	where $p\in L_{n}$ satisfies $pK_{n}=gsK_{n}$. In the last equality, we have used that $gs\in Q_{n}$. Now, since 
	\[ \phi_{\gamma}\tilde{\sigma}_{s^{-1}}(gK_{n})|_{K_{n+1}}=\phi_{\gamma}({pK_{n})}|_{K_{n+1}}=p^{-1}({w_{\gamma}^{n}})^{-1}x|_{K_{n+1}}=x|_{w_{\gamma}^{n}pK_{n+1}}, \]we conclude that $d(\phi_{\gamma}\tilde{\sigma}_{s^{-1}}(gK_{n}),s^{-1}\phi_{\gamma}(gK_{n}))=0$. Summing up, since $g\in L_{n}$ and $s\in F$ were arbitrary, for any $\gamma\in k^{\Wee_{n}}$ we have $d_{2}(\phi_{\gamma}\tilde{\sigma}_{s^{-1}},s^{-1}\phi_{\gamma})=0$ and $\phi$ is an $(F\cup F^{-1},\delta)$-pseudoorbit. Note that for distinct $\gamma,\gamma'\in k^{\Wee_{n}}$ their respective pseudoorbits $\phi_{\gamma}$ and $\phi_{\gamma'}$ are $(\epsilon,d_{\infty})$-separated. Indeed, let $\gamma'(s)\neq\gamma(s)$ for some $s\in\Wee_{n}$, then \[ \phi_{\gamma}(sK_{n})|_{K_{n+1}}=x_{w_{\gamma}^{n}K_{n+1}s}=\gamma(s) \] and
	\[\phi_{\gamma'}(sK_{n})|_{K_{n+1}}=x_{w_{\gamma'}^{n}K_{n+1}s}=\gamma'(s), \]so $d_{\infty}(\phi_{\gamma},\phi_{\gamma'})\geq1$.
	As a consequence $\{\phi_{\gamma}\}_{\gamma\in k^{\Wee_{n}}}\subset \map(d,F\cup F^{-1},\delta,\tilde{\sigma}_{n})$ is an $(\epsilon,d_{\infty})$-separated set for every $\epsilon>0$. We can finally estimate
	\[ 
	k^{a_{\beta_{n}}}=k^{|\Wee_{n}|} \leq N_{\epsilon}(\map(d,F\cup F^{-1},\delta,\tilde{\sigma}_{n}),d_{\infty})\leq N_{\epsilon}(\map(d,F,\delta,\tilde{\sigma}_{n}),d_{\infty}),
	 \]taking the exponential growth of both sides and applying the appropriate limits, we obtain 
	 \[ 
	 \lim_{n\to\infty}\frac{a_{n}}{[G:H_{n}]}\log k=\lim_{n\to\infty}\frac{a_{\beta_{n}}}{[G:H_{\beta_{n}}]}\log k=\lim_{n\to\infty}\frac{a_{\beta_{n}}}{|L_{n}|}\log k \leq h_{\tilde{\Sigma}}(X,G)=h_{\Sigma}(X,G),
	  \]which proves the theorem.\qedhere
	
\end{proof}
At last we proceed to the Krieger's Theorem. The amenable counterpart, proved by direct construction, can be found in \cite{KriegerTheorem}. On the other hand, indirect proof is presented in \cite{MM2018}.
\begin{thm}\label{krieger-rf}
	Let $G$ be a residually finite group and $H_{n}\searrow\{e\}$ be a decreasing sequence of finitely indexed normal subgroups of $G$. Then for every $k\in\mathbb{N}$ bigger then 1 and every number $\kappa \in [0,1)$ there exists a Toeplitz element $x\in k^{G}$ with  $h_{\Sigma}(\overline{Gx},G) = \kappa \log k$, where $\Sigma=\{\sigma_{n}\}_{n\in\nat}$ is a sofic approximation sequence defined by the natural action on cosets of $H_{n}$ for every $n\in\nat$.
\end{thm}
\begin{proof}
	For $n\in\mathbb{N}$, consider division of $[0,1)$ into finite number of disjoint intervals, each of length $1/[G:H_n]$. Find such an $a_n\in\mathbb{N}$, that $\kappa \in [\frac{a_n-1}{[G:H_n]},\frac{a_n}{[G:H_n]})$. Define in this way sequence $\{a_n\}_{n=1}^{\infty}$. If necessary, we can take subsequence of $\{a_n\}_{n=1}^{\infty}$ and consider it with a proper subsequence of groups, so that $a_n<a_{n-1}[H_{n-1}:H_{n}]$ is satisfied. Clearly, using Theorem \ref{toeplitz-example} we are able to construct a Toeplitz element $x\in k^{G}$, such that $h_{\Sigma}(\overline{Gx},G) = \kappa \log k$.
\end{proof}
\textbf{Acknowledgements} This article was an author's master degree thesis. It was supervised by Dominik Kwietniak, who put a lot of effort to make it readable, for which we are utterly grateful.
\bibliography{mybib}
\bibliographystyle{alpha}

\end{document}